\documentclass[a4paper,12pt,reqno]{amsart}
 \usepackage[english]{babel}
  \usepackage{amsmath,amsfonts,amssymb,amsthm,bbm,textcomp}
   \usepackage{graphics,epsfig,psfrag} 
   \usepackage{paralist,subfigure,multirow,float,scrextend,esint}

	\usepackage{hyperref} 
    \usepackage[active]{srcltx} 
    \usepackage{color,soul} 
    \usepackage[latin1]{inputenc}
    \usepackage[OT1]{fontenc}

\newtheorem{theorem}{Theorem}[section]
\newtheorem{corollary}[theorem]{Corollary}
\newtheorem{lemma}[theorem]{Lemma}
\newtheorem{proposition}[theorem]{Proposition}

\newtheorem{open}{Open problem}

\theoremstyle{definition}
\newtheorem{definition}[theorem]{Definition}
\newtheorem{remark}{Remark}

\DeclareMathOperator{\dist}{dist}
\DeclareMathOperator{\dv}{div}

\DeclareMathOperator\supp{supp}
\DeclareMathOperator\Tr{Tr}
\DeclareMathOperator\tr{tr} 

\def\cprime{$'$}

\def\N{\mathbb{N}}
\def\Z{\mathbb{Z}}

\def\R{\mathbb{R}}

\let\O=\Omega
\let\e=\varepsilon
\let\vp=\varphi

\let\t=\tilde
\let\ol=\overline
\let\ul=\underline
\let\.=\cdot
\let\0=\emptyset

\let\mc=\mathcal

\def\notimplies{\mathrel{{\ooalign{\hidewidth$\not\phantom{"}$\hidewidth\cr$\implies$}}}}

\def\1{\mathbbm{1}}

\newcommand{\su}[2]{\genfrac{}{}{0pt}{}{#1}{#2}}

\makeatletter
\newcommand*{\xRightarrow}[1][]{%
  \ext@arrow 0359\Rightarrowfill@{\quad\!\;}{#1}%
}
\makeatother

\def\thm#1{Theorem~\ref{thm:#1}}
\def\seq#1{(#1_n)_{n\in\N}}

\def\limt{\lim_{t\to+\infty}}
\def\as{\quad\text{as }\;}
\def\inn{\quad\text{in }\;}
\def\on{\quad\text{on }\;}
\def\for{\quad\text{for }\;}
\def\fall{\quad\text{for all }\;}
\def\step#1{\smallskip{\em Step #1.\,}}

\newenvironment{formula}[1]{\begin{equation}\label{#1}}
                       {\end{equation}\noindent}

\def\Fi#1{\begin{formula}{#1}}
\def\Ff{\end{formula}\noindent}

\newenvironment{nota}[1]{\textcolor{red}
*\marginpar{\footnotesize{\textcolor{red}{#1}}}}

\def\pe{principal eigenvalue}
\def\pf{principal eigenfunction}
\def\MP{maximum principle}
\def\SMP{strong maximum principle}

\def\L{\mathcal{L}}

\def\l{\lambda}

\def\lB{\lambda_{\mc{B}}}
\def\LB{\Lambda_{\mc{B}}}
\def\B{\mc{B}}

\def\IP{\textrm{\bf{IP}}}

\def\mean#1{\left\langle#1\right\rangle}
\def\lm#1{\left\lfloor{#1}\right\rfloor}
\def\um#1{\left\lceil{#1}\right\rceil}

\usepackage{tabularx,ragged2e,booktabs,caption}

\usepackage{pdfrender}

\usepackage{amsbsy}

\title[Stability analysis in general unbounded domains]
{\bf Stability analysis for semilinear parabolic problems in general unbounded domains}
\author{Luca  Rossi}
\address{CNRS, Ecole des Hautes Etudes en Sciences Sociales,   
Centre d'Analyse et Math\'ematiques Sociales, 54, boulevard Raspail, 
75006 Paris, France}
\email{luca.rossi@ehess.fr}
\begin{document}

\maketitle

\vspace{-2pt}

\begin{abstract}
	
We introduce several notions of generalised \pe\ for
a linear elliptic operator on a general unbounded domain, under boundary condition
of the oblique derivative type.
We employ these notions in the stability analysis of semilinear problems.
Some of the properties we derive are new even in the Dirichlet or in the whole space cases. 
As an application, we show the validity of the hair-trigger effect for 
the Fisher-KPP equation on general, uniformly smooth~domains.
%
\end{abstract}



\section{Introduction}

The prototype of reaction-diffusion equations is
\begin{equation}\label{RD}
\partial_{t} u=\Delta u+f(u), \quad  t>0,  \ x\in\mathbb{R}^{d}.
\end{equation}
It has been introduced independently by Fisher \cite{Fisher} and
Kolmogorov, Petrovski and Piskunov \cite{KPP} for modeling the propagation 
of a gene inside a population. It has then been employed in several contexts, ranging from
population dynamics, epidemiology, combustion theory to social sciences.
The question at stake is whether a quantity subject to diffusion and self-reinforcement
will eventually invade the whole environment and at which speed.
Invasion means that the solution $u$ becomes uniformly bounded from below away from zero 
as $t\to+\infty$
(a precise definition will be given in the sequel).
The result of~\cite{KPP}, completed in~\cite{AW}, is that invasion occurs for any nonnegative, nontrivial initial datum,
with a positive asymptotic speed, at least as soon as $f$ satisfies $f'(0)>f(0)=0$. 
This property is referred to as the ``hair-trigger" effect.
There has been a huge amount of improvements and extensions of this result
both for modelling and for purely theoretical purposes. Just to mention very few of them:
more general reaction terms $f$ have been considered in~\cite{AW} 
--where the condition $f'(0)>0$ is relaxed by $\liminf_{s\to0^+}s^{-p}f(s)>0$, 
with $p=1+2/d$ being the Fujita exponent--
spatial periodic heterogeneity
in~\cite{W02,BHR1,BHNperiodic,Holes}, random stationary ergodic coefficients in~\cite{NX09,FG}, 
general spatial heterogeneity in~\cite{BHRossi,BHNgeneral}. 

Equation~\eqref{RD} is used to model the dynamics of a population under the 
assumption that 
the habitat is homogeneous and coincides with the whole plane.
If instead the environment presents some obstacles (such as bodies of water,
walls, mountains,\,...) then one is led to consider the Neumann problem
\Fi{RDN}
\begin{cases}
	\partial_t u=\Delta u+f(u), & \quad t>0,\ x\in\O\\
	\nu\.\nabla u=0, & \quad t>0,\ x\in\partial\O,
\end{cases}
\Ff
where $\nu$ is the normal field to $\O$. The Neumann boundary condition is very natural:
it is a no-flux condition which entails the conservation of mass 
(total population) in the absence of the reaction term $f$.
There is a vast literature on this problem when the domain $\O$ is either bounded or it
has a specific shape (cylinder, periodic domain, exterior domain,\,...).
To the best of our knowledge, the only result for problem~\eqref{RDN} 
set on a generic unbounded domain $\O$ 
is \cite[Theorem~1.7]{BHNgeneral}\footnote{
\;\cite[Theorem~1.7]{BHNgeneral} contains some additional hypotheses,
but they are only required to ensure that~$\O$ is ``strongly unbounded'' in any direction,
a property used to define the asymptotic speed of~spreading.}
by Berestycki, Hamel and Nadirashvili.
%
%
\begin{theorem}[\cite{BHNgeneral}]\label{thm:BHNgeneral}
Assume that $f'(0)>f(0)=0$ and that $\O$ is an unbounded domain such that $\partial\O$ is
uniformly of class $C^{2,\alpha}$ and the following conditions hold:
\begin{enumerate}[\quad $a)$]
	\item the geodesic distance in $\ol\O$, denoted by $\delta_\O$, satisfies 
	$$\forall r>0,\quad
	\sup_{\su{x,y\in\O}{|x-y|<r}}\delta_\O(x,y)<+\infty;$$
	\vspace{-10pt}
	\item $\displaystyle\lim_{r\to+\infty}\frac{|\O\cap B_{r+1}(y)|}{|\O\cap B_r(y)|}=1$,
	uniformly with respect to $y\in\O$.
\end{enumerate}
Then any solution to~\eqref{RDN} with an initial condition $u_0\geq0,\not\equiv0$
satisfies
$$\inf_{x\in\O}\Big(\liminf_{t\to+\infty} u(t,x)\Big)>0.$$
\end{theorem}
This result ensures that invasion always occurs if the habitat 
fulfils the geometric hypotheses $a,b)$.
What happens if one drops such hypotheses stands as an open question. 
One of the goals of this paper is to answer it. 
Examples violating $a)$ are: spiral domains, the complement of a half-hyperplane; 
counter-examples to~$b)$ are less standard, 
but still not so hard to construct,
see Remark~\ref{rem:BHNgeneral} in Section~\ref{sec:s-a}. 

The result of \thm{BHNgeneral} is related to an instability
feature of the null state. To investigate such phenomenon, 
we make use of the notion of {\em generalised \pe}, that we apply
to the linearised problem around $0$. 

The theory of the generalised \pe\ in an unbounded
domain under Neumann boundary condition is less developed than the corresponding one for
the Dirichlet condition, 
despite the former is very natural to consider in modelling, as mentioned before.
One of the reasons behind this fact resides in some technical difficulties in the approximation
of the generalised \pe\ through classical \pe s, 
that will be explained in Section~\ref{intro:lB}.
In addition to the properties needed to improve \thm{BHNgeneral}, we investigate some 
further features of the generalised \pe\ 
for an arbitrary linear elliptic operator in non-divergence form,
with the aim of filling 
the gap between the known results about the Dirichlet
and the general oblique derivative conditions, which include Neumann as a particular case.
These features are:
approximation of the generalised \pe\ through 
truncated domains, existence of positive eigenfunctions, connections with the stability analysis.
In the work~\cite{BR4} in collaboration with Berestycki, we have shown that in the Dirichlet case
a unique notion of \pe\ does not suffice to characterise all these properties.
As we will see, the same occurs in the oblique derivative case. 
The relations between the different notions of generalised \pe\ introduced are then discussed.



\section{Main results}\label{sec:main}

We consider the linear operator 
$$\L w:=\Tr(A(x)D^2 w)+b(x)\.\nabla w+c(x)w,\quad x\in\O.$$
The matrix field $A$ is assumed to be symmetric and elliptic, with the
smallest ellipticity constant $\ul A(x)$ being positive (not necessarily uniformly)
for every $x\in\ol\O$.
The basic regularity assumptions are $A,b,c\in C^{0,\alpha}_{loc}(\ol\O)$ for some $\alpha\in(0,1)$.

The regularity of the boundary
of the domain~$\O\subset\R^d$ will play a crucial role in our results. 
We recall that $\partial\O$ is said to be
{\em locally of class}~$C^k$ if at each point $\xi\in\partial\O$
there correspond a coordinate system 
$(x', x_n)\in\R^{d-1}\times\R$, 
a constant $\rho>0$ and a function $\Psi : \R^{d-1}\to\R$ of class~$C^k$ 
such that 
$$ B_\rho(\xi) \cap\O = \{(x', x_n) \ :\ x_n > \Psi(x')\} \cap B_\rho(\xi).$$
If $\rho$ can be chosen independently of
$\xi\in\partial\O$ and $\|\Psi\|_{C^k}\leq C$ for some $C$ independent of 
$\xi$ too, then $\partial\O$ is said to be
{\em uniformly of class}~$C^{k}$. 

We assume throughout the paper that $\partial\O$ is  locally of class~$C^d$, or~$C^{2,\alpha}$
in the case~$d=2$.
This allows one to approximate $\O$ with a family of ``truncated'' Lipschitz domains,
hence to treat elliptic and parabolic problems in $\O$ as the limits of bounded domains.
The $C^d$ regularity is a technical condition
required to perform such an approximation through an application
of the Morse-Sard theorem to the distance function to $\partial\O$ (which inherits
the regularity of $\partial\O$, see~\cite[Lemma~14.16]{GT}).

As for the boundary condition on $\partial\O$, it is given by the oblique derivative operator
$$\B w:=\beta(x)\.\nabla w+\gamma(x) w,$$
with $\beta:\partial\O\to\R^d$ and $\gamma:\partial\O\to\R$ in $C^{1,\alpha}_{loc}$.
The vector field $\beta$ is assumed to satisfy
$$\beta\.\nu>0\quad\text{on }\partial\O,$$
where $\nu$ is the exterior normal field to $\O$.


\subsection{The ``standard'' generalised principal eigenvalue}\label{intro:lB}

The ``standard'' notion of generalised \pe\ is
$$\lB:=
\sup\{\lambda\ :\ \exists\phi>0,\ (\L+\lambda)\phi\leq0 \text{ in }
\O,\ \B \phi\geq0\text{ on }\partial \O\}.
$$
The ``test functions'' $\phi$ are assumed to belong to $C^{2,\alpha}_{loc}(\ol\O)$.
In the Dirichlet case we have $\B\phi=\phi$, therefore
 the condition $\B\phi\geq0$ is automatically fulfilled.
Indeed, in such a case, the above definition is exactly the same as the one introduced
by Berestycki, Nirenberg and Varadhan in~\cite{BNV}
to deal with bounded, non-smooth domains $\O$ 
(earlier equivalent definitions are due to Agmon~\cite{A1}
in the case of operators in divergence form 
and, for general operators, to Nussbaum and Pinchover 
\cite{NP}, building on some ideas of Protter and Weinberger 
\cite{Max2}).
The definition of~$\lB$ for an oblique derivative operator~$\B$ with~$\gamma\geq0$
can be found in~\cite{P-S02}.\footnote{\;The definition given in~\cite{P-S02}
 is actually slightly different, but it is shown to be 
equivalent to the above one.} We also refer to the work of Patrizi~\cite{Patri08}
where the same notion is applied to a class of homogeneous, fully nonlinear operators in
bounded, smooth domains.


The key property of $\lB$ that we derive is that it can be approached by ``classical'' \pe s in
truncated domains.
Namely, for a given point~$y\in\O$,
we let $\O_r(y)$ denote the connected component of $\O\cap B_r(y)$ 
containing $y$, and we consider the eigenvalue problem
\Fi{ep-mixed}
\begin{cases}
	-\L\vp=\lambda\vp & \text{ in }\O_r(y)\\
	\B \vp=0 & \text{ on } (\partial \O_r(y))\cap B_r(y)\\
	\vp=0 & \text{ on }\partial \O_r(y)\cap \partial B_r(y).
\end{cases}
\Ff
This is a mixed boundary value problem in a domain that
can be very irregular at the intersection between the boundaries 
of $\O$ and $B_r(y)$. 
This difficulty is bypassed in the Dirichlet case by approximating 
the truncated domains from inside 
by a family of smooth domains, then exploiting in a crucial way the monotonicity of the 
\pe\ with respect to the inclusion of domains. Such monotonicity fails in the Neumann case,
hence here we need to directly work on the irregular domain, facing the related technical 
difficulties.
We overcome it by means of a result obtained in collaboration with Ducasse~\cite{Holes}
which ensures that $\O\cap B_r(y)$ is a Lipschitz open set for a.e.~$r>0$,
allowing us to apply the solvability theory of Lieberman~\cite{Lie86}.

\begin{theorem}\label{thm:lrinfty}
Let $y\in\O$. For a.e.~$r>0$, the eigenvalue problem~\eqref{ep-mixed} admits a unique  
eigenvalue $\lambda(y,r)$ with an eigenfunction $\vp$ which is positive in $\ol\O_r\cap B_r$.

Moreover, $\lambda(y,r)$ is 
strictly decreasing with respect to $r$ and satisfies%
\footnote{\;Here and in the sequel, we extend $r\mapsto\lambda(y,r)$
	to a semicontinuous (decreasing) function, so that the limit at infinity makes sense.}
$$\lim_{r\to\infty}\lambda(y,r)=\lB.$$
\end{theorem}

The next result establishes a link between the
generalised \pe\ $\lB$ and the existence 
of positive eigenfunctions for the oblique derivative~problem.
\begin{theorem}\label{thm:ef}
If $\O$ is unbounded then the eigenvalue problem
\Fi{epB}
\begin{cases}
	-\L\vp=\l\vp & \text{ in }\O\\
	\B \vp=0 & \text{ on } \partial \O,
\end{cases}
\Ff
admits a positive solution $\vp$ if and only if $\l\leq\lB$.
\end{theorem}

The same result is derived  in~\cite{P-S02} in the case $\gamma\geq0$ 
through an elaborate 
application of Perron's method. We use a different approach based on the approximation
by truncated domains, provided by \thm{lrinfty}.
Such an approximation
is also our key tool to perform the stability analysis and to tackle the open question related 
to \thm{BHNgeneral}.
\thm{ef} enlightens a major difference with respect to 
the classical case of a bounded smooth domain: there exist several (a half-line of)
eigenvalues associated with a positive eigenfunction.
Another difference is that~$\lB$ is not simple in general, see \cite[Proposition~8.1]{BR4}
for an example with $\O=\R$ and $\L$ self-adjoint.
In the sequel, any positive eigenfunction $\vp$ associated with the eigenvalue $\lB$ will be called a
 {\em generalised principal eigenfunction}. Hopf's lemma implies that $\vp>0$ in $\ol\O$.


\subsection{Stability analysis}\label{intro:stability}

Consider the semilinear problem
\Fi{BP}
\begin{cases}
	\partial_t u=\Tr(A(x)D^2 u)+b(x)\.\nabla u+f(x,u), & \quad t>0,\ x\in\O\\
	\mc{B} u=0, & \quad t>0,\ x\in\partial\O.
\end{cases}
\Ff
We investigate the stability of steady states -- i.e., time-independent solutions --
for the Cauchy problem associated with~\eqref{BP}. 
Throughout the paper, the initial data~$u_0$ are assumed to be continuous and
solutions are obtained as limits as $r\to+\infty$ of Cauchy problems in the truncated 
(Lipschitz) domains $\O_r$, by imposing the Dirichlet condition $u=u_0$ on the new boundary portion
$\partial B_r\cap\partial \O_r$, to which the solvability theory of~\cite{Lie86} applies.
We discuss several notions of stability.
The first one is that of asymptotic stability, 
that for short we simply refer to as ``stability''.
\begin{definition}\label{def:stable}
	We say that a steady state $\bar u$ of~\eqref{BP} is {\em uniformly stable}
	(resp.~{\em locally stable}) if there exists $\delta>0$ 
	such that, for any initial datum $u_0$
	with $\|u_0-\bar u\|_\infty<\delta$, the associated solution $u$ satisfies
	$$u(t,x)\to\bar u(x)\as t\to+\infty,\quad\text{uniformly (resp.~locally uniformly)
		in $x\in\ol\O$}.$$
	We say that $\bar u$ is locally/uniformly stable {\em with respect to
		compact perturbations}
	if the above properties hold provided $u_0=\bar u$ outside a compact subset of $\ol\O$.
\end{definition}

A notion of strong instability which often arises in application is the following.
\begin{definition}\label{def:repulsive}
	We say that a steady state $\bar u$ of~\eqref{BP} is {\em uniformly repulsive}
	(resp.~{\em locally repulsive}) (from above) if for any solution $u$ with an
	initial datum $u_0\geq \bar u,\not\equiv\bar u$, the function 
	$$u_\infty(x):=\liminf_{t\to+\infty}u(t,x),$$
	satisfies
	$$\inf_{x\in\O}\big(u_\infty(x)-\bar u(x)\big)>0\qquad \text{(resp.~$u_\infty>\bar u$ in 
	$\ol\O$)}.$$
\end{definition}
Then the conclusion of \thm{BHNgeneral} can be rephrased as the 
uniform repulsion of the null state. 
It implies in particular that the problem does not admit
any positive supersolution with zero infimum.

We remark that, up to replacing the semilinear term $f$ with
$$\t f(x,s):=f(x,\bar u(x)+s)-f(x,\bar u(x)),$$
we can always reduce to study the case where the steady state is
$\bar u\equiv0$. 
So, throughout the paper, we will assume that 
$$f(x,0)=0\fall x\in\O.$$
Regularity assumptions on $f$ are that
$s\mapsto f(x,s)$ is of class $C^1$
in a neighbourhood of~$0$ as well as locally Lipschitz-continuous on $\R_+$,
uniformly with respect to $x\in\O$.
We further assume that $f(x,s),f_s(x,0)\in C^{0,\alpha}_{loc}(\ol\O)$, locally uniformly
with respect to~$s$.
The assumptions on $A$ and $b$ are the same as before.

The stability analysis for~\eqref{BP} relies on the sign of the generalised \pe\ associated
with the linearised operator around $0$, that is,
$$\L w:=\Tr(A(x)D^2 w)+b(x)\.\nabla w+f_s(x,0)w.$$
As a matter of fact, the notion of generalised \pe\ needs to be adapted 
to the kind of stability property one is looking at,
the quantity $\lB$ introduced in the previous section being responsible only for some them. 
For instance, $\lB>0$ does~not imply that $0$ is locally stable, not even with respect to
compact perturbations, see Propositions~\ref{pro:ce-stability} below.
By changing the ``test-functions'' $\phi$ 
in the definition of $\lB$, we end up with four other quantities 
(c.f.~\cite{BHRossi,BR4} for the Dirichlet~case):
$$\lB^p:=
\sup\{\lambda\ :\ \exists\phi,\ \inf_\O\phi>0,\ (\L+\lambda)\phi\leq0 \text{ in }
\O,\ \B\phi\geq0\text{ on }\partial \O\},
$$
$$\lB^b:=
\sup\{\lambda\ :\ \exists\phi>0,\ \sup_\O\phi<+\infty,\ (\L+\lambda)\phi\leq0 \text{ in }
\O,\ \B\phi\geq0\text{ on }\partial \O\},
$$
$$\lB^{p,b}:=
\sup\{\lambda\ :\ \exists\phi,\ \inf_\O\phi>0,\ \sup_\O\phi<+\infty,\ (\L+\lambda)\phi\leq0 \text{ in }
\O,\ \B\phi\geq0\text{ on }\partial \O\},
$$
$$\mu_{\mc{B}}^b:=
\inf\{\lambda\ :\ \exists\phi>0,\ \sup_\O\phi<+\infty,\ (\L+\lambda)\phi\geq0 \text{ in }
\O,\ \B\phi\leq0\text{ on }\partial \O\}.
$$
We derive necessary and sufficient conditions for the stability in terms of all these notions.
These are summarised in the following table, which contains the results of
Propositions~\ref{pro:stability} and~\ref{pro:locrepulsive} below.
We point out that the table holds true in the Dirichlet case $\B u=u$ 
(up to relaxing the local repulsion by $u_\infty>\bar u$ in $\O$).

\noindent
\begin{minipage}[t]{\linewidth}
	\centering
	
	\captionof{table}{{\em Stability properties}\vspace{-8pt}} \label{tab:stability} 
	\begin{tabular}{|l c  c c l|}
	\hline \vspace{-5pt}&&&&\\ 
		
		$\;\lB^b>0$ & $\implies$ 
		&\!$\begin{array}{cc}\text{local/uniform stability}\\ 
		\text{w.r.t.~compact perturbations}\end{array}$\!& 
		$\implies$ & $\lB\geq0$\\ 
		&&&&\\
		
		$\;\lB^{p,b}>0$ & $\implies$ &  local/uniform stability & $\implies$ & 
		$\mu_{\mc{B}}^b\geq0$\;\\
		&&&& \\
		
		$\;\lB<0$ & $\implies$ & local repulsion & $\implies$ & $\lB^{b}\leq0$\;\\
		\vspace{-0.5pt}&&&&\\ \hline
%
	\end{tabular}\par
	\smallskip
	\bigskip
\end{minipage}

Propositions \ref{pro:ce-stability} below shows through explicit
counter-examples that the implications in this table cannot
be improved by replacing the involved notion of generalised \pe\ with a different one.
Table~\ref{tab:stability} does not contain any condition guaranteeing the
uniform repulsion. 
Indeed, even the negativity of $\lB$, which is the largest quantity among the ones defined before
(see~\thm{relations} below),
does not prevent the existence of positive steady states with zero infimum, 
as it is the case for instance of the equation considered in~\cite{BR2}.
The reason is that, by \thm{lrinfty}, $\lB<0$ implies that $\lambda(y,r)<0$ for $r$ 
sufficiently large, which entails the local repulsion,
but in order to capture the behaviour of solutions at infinity one needs 
this property to hold with $r$ independent of $y$. This leads us to define
the following quantity:
$$\LB:=\lim_{r\to+\infty}\bigg(\sup_{y\in\O}
\lambda(y,r)\bigg).$$
Observe that this only differs from $\lB$ because of the supremum with respect to~$y\in\O$ 
(taking the infimum would give again $\lB$).
Roughly speaking, 
$\LB<0$ means that the \pe s in large truncated domains are ``uniformly negative''.
Under this condition, together with the uniform bounds
\Fi{hyp:uni}
A,b\in L^\infty(\O),\qquad
\inf_{\O}\ul A>0,\qquad
\beta,\gamma\in L^\infty(\partial\O),\qquad
\inf_{\partial\O}\beta\.\nu>0,
\Ff
(where, we recall, $\ul A(x)$ is the smallest ellipticity constant of $A(x)$)
we derive the uniform repulsion of the steady state.
\begin{theorem}\label{thm:LambdaB}
	Assume that $\partial\O$ is 
	uniformly of class $C^2$ and that~\eqref{hyp:uni} holds. 
	If $\LB<0$ then~$0$ is uniformly repulsive.
\end{theorem}

We conclude this section discussing a global stability notion which
is important in the context of population dynamics to determine whether
extinction occurs no matter the initial size of the population. 
\begin{definition}\label{def:attractive}
	We say that the null state is {\em globally attractive}
	if any solution $u$ with a bounded initial datum $u_0\geq0$ satisfies
	$$u(t,x)\to0\as t\to+\infty,\quad\text{locally uniformly
		in $x\in\ol\O$}.$$
	We say that the null state is globally attractive {\em with respect to
		compact perturbations}
	if the above property holds provided $u_0$ is compactly supported in
	$\ol\O$.
\end{definition}
We point out that, unlike the previous notions, the global attraction cannot be linearly determined, but some condition on the nonlinear term far from the null state needs to be imposed. Having in mind applications to population dynamics, we consider here the Fisher-KPP hypothesis of
\cite{Fisher,KPP}:
\Fi{KPP}
\forall x\in\O,\ s>0,\quad f(x,s)\leq f_s(x,0)s,
\Ff
that will sometimes be combined with the saturation condition
\Fi{saturation}
\gamma\geq0,\qquad
\exists S>0,\ \forall x\in\O,\ s\geq S,\quad f(x,s)\leq 0.
\Ff 
The former entails that positive solutions to the linearised problem are supersolutions
to~\eqref{BP}, while the latter implies that large constants are. 
The sufficient conditions for the attraction are summarised in the following.

\noindent
\begin{minipage}[t]{\linewidth}
	\centering
	
	\captionof{table}{{\em Attraction}\vspace{-7pt}} \label{tab:attraction} 
	\begin{tabular}{ | r c l | }
			\hline \vspace{-5pt}&&\\ 
	 \eqref{KPP}, \; $\lB>0$ & $\implies$ 
	 &$\hspace{-5pt}\begin{array}{ll}\text{global attraction w.r.t.}\\ 
	 \text{compact perturbations}\end{array}$\!\\ 
	 &&\\
	 
	 \eqref{KPP}, \; $\lB^p>0$ & $\implies$ & global attraction \\ 
	 &&\\
	 
	 \eqref{KPP}-\eqref{saturation},\; 
	 $\mu_{\mc{B}}^b>0$ & $\implies$ & global attraction \\
	\vspace{-.5pt}&&\\ 
	\hline  
	\end{tabular}\par
\bigskip
\end{minipage}

Some of the results contained in Tables \ref{tab:stability} and \ref{tab:attraction} 
are new even in the case $\O=\R^d$.


\subsection{The self-adjoint case}\label{intro:s-a}

Next, we focus on the self-adjoint Neumann problem. That is, we take
\Fi{LBs-aN}
\L w:=\nabla\cdot (A(x) \nabla w)+c(x)w,\qquad
\B w:=\nu \.A(x)\nabla w.
\Ff
We assume in such case that $A\in C^{1,\alpha}_{loc}(\ol\O)$, so that 
$\L$, written in non-divergence form, satisfies our standing hypotheses.
We derive an upper bound for $\lB$ and $\LB$ in terms of the 
following notions of average of $c$:
$$\mean{c}:=\liminf_{r\to+\infty}\,\frac{\int_{\O_r(y)}c}
{\left|\O_r(y)\right|},$$
$$\lm{c}:=\liminf_{r\to+\infty}\,\left(\inf_{y\in\O}\frac{\int_{\O_r(y)}c}
{\left|\O_r(y)\right|}\right).$$
The latter quantity is related to the notion of {\em least mean} introduced in~\cite{NR1}.

\begin{theorem}\label{thm:l<<0}
	Let $\L$,~$\B$ be given by~\eqref{LBs-aN} with $A,c$ bounded.
	There holds that 
	$$\lB\leq-\mean{c}.$$
	
	If, in addition, $\O$ satisfies the uniform interior ball condition, then
	$$
	\LB\leq-\lm{c}.
	$$
\end{theorem}
The uniform interior ball condition means that at each point of the boundary there exists
a tangent ball contained in $\O$, with a radius independent of the point; this is guaranteed
if $\partial\O$ is uniformly of  class~$C^{1,1}$, see, e.g., \cite[Lemma 2.2]{C11}.

Gathering together Theorems~\ref{thm:LambdaB} and~\ref{thm:l<<0} we
finally improve \thm{BHNgeneral} by dropping the hypotheses $a,b)$.
We also allow a general diffusion term, that is, 
\Fi{BPs-a}
\begin{cases}
	\partial_t u=\nabla\cdot (A(x) \nabla u)+f(u), & \quad t>0,\ x\in\O\\
	\nu\.A(x)\nabla u=0, & \quad t>0,\ x\in\partial\O.
\end{cases}
\Ff
\begin{corollary}\label{cor:h-t}
	Assume that $A$ is bounded and uniformly elliptic, that
	$f'(0)\!>\!f(0)\!=\!0$ and that $\partial\O$
	is uniformly of class $C^2$.
	Then any solution to~\eqref{BPs-a} with an initial condition $u_0\geq0,\not\equiv0$
	satisfies
	$$\inf_{x\in\O}\Big(\liminf_{t\to+\infty} u(t,x)\Big)>0.$$
\end{corollary}
The above result actually holds true for $f=f(x,s)$ satisfying $\lm{f_s(x,0)}>0$ (but fails 
in general under the condition $\mean{f_s(x,0)}>0$,
see Proposition~\ref{pro:<c>} below);
if, in addition, $f(x,s)$ is strictly concave, uniformly with respect to $x\in\O$,
 then one readily deduces that any 
nontrivial solution converges to the unique positive steady state as $t\to+\infty$. The same 
is true if $f(x,s)$ is strictly positive for $s\in(0,1)$ and negative for $s>1$,
c.f.~Proposition~\ref{pro:h-t}.


We stress out that the validity of the hair-trigger effect for non-uniformly
smooth domains remains an open question.

\begin{remark}
It is clear that the (local) repulsion of $0$ for the Fisher-KPP
equation 
$$\partial_t u=\Delta u+u(1-u)$$
may fail in the Dirichlet case $\B u=u$.
For instance, in the strip $\O=\R\times(-1,1)$ we have that $\lB^{p,b}=\frac{\pi^2}4-1$
and thus $0$ is uniformly stable, owing to Table~\ref{tab:stability}. 
For the equation in the same strip, but with Robin boundary condition
$\B u=\partial_\nu u+\gamma u$, we have that $\LB=\lB^{p,b}<0$ (hence $0$ is uniformly repulsive)
if and only if $\gamma<\tan1$.
By Corollary~\ref{cor:h-t},
the same bifurcation phenomenon occurs 
for a general unbounded, uniformly smooth domain $\O$, 
at some threshold value $\gamma_0\in(0,+\infty]$.
\end{remark}


\subsection{Relations between the 
generalised \pe s}\label{sec:MP}
In view of the previous results,
it is very useful to know the relations between the different notions of generalised \pe.
They are summarised in the following theorem, which also provides the equivalence
between $\mu_\B^b$ and $\lB$ in the self-adjoint case under
Robin boundary condition (no sign assumption on $\gamma$):
\Fi{LBs-a}
\L w:=\nabla\cdot (A(x) \nabla w)+c(x)w,\qquad
\B w:=\nu \.A(x)\nabla w+\gamma(x)w.
\Ff
\begin{theorem}\label{thm:relations}
	Assume that the coefficients of $\L$ and $\B$ are bounded. The following relations hold:
	\begin{enumerate}[$(i)$]
		\item \
		\vspace{-15pt}
		$$\lB^{p,b}\leq\min\{\lB^p,\lB^b\}\leq\max\{\lB^p,\lB^b\}\leq\lB\leq\LB\,;$$
		\item  \
		\vspace{-15pt}
		$$\lB^{p,b}\leq\mu_\B^b\leq\lB\,;$$ 
		\item if $\L$, $\B$ are in the self-adjoint form~\eqref{LBs-a} then
		$$\mu_\B^b=\lB\,;$$
		\item if 
		either~\eqref{beta>0} or~\eqref{gamma>0} below hold then
		$$\lB^{p}\leq\mu_\B^b.$$
	\end{enumerate}
\end{theorem}
The inequalities in $(i)$ immediately follow from the definitions.
The other ones are nontrivial. Actually, the boundedness of the coefficients can~be relaxed, depending on 
the case, see Remark~\ref{rem:relax} below.
The hypotheses for $(iv)$ are~either
\Fi{beta>0}
\partial\O\text{\, is uniformly of class $C^2$ \ and \ \,}
\inf_{\partial\O}\beta\.\nu>0,
\Ff
or
\Fi{gamma>0}
\inf_{\partial\O}\gamma>0.
\Ff
These are reasonable hypotheses which ensure, for instance, the local 
well posedness for the parabolic problem~\eqref{BP}.
Statement~$(iv)$ has some useful consequences. 
The first one is that~$\lB^{p}>0$ is 
a sufficient condition for the validity of the {\em\MP} (i.e., bounded 
subsolutions of $\L=0$, $\B=0$ are necessarily nonpositive).
Indeed, the \MP\ holds if $\mu_\B^b>0$ and only if $\mu_\B^b\geq0$, 
as a consequence of the fact that the existence of nonnegative 
subsolutions or of positive subsolutions are equivalent, as shown
%
in a work in progress by~Nordmann.
Another immediate consequence of $(iv)$ is that in the periodic case, i.e.~if the domain~$\O$
as well as the coefficients of $\L$ and $\B$ are periodic with the same period, 
the quantities $\lB^p,\lB^{p,b},\mu_\B^b$ coincide with the periodic \pe. Then, in such case,
we recover from Tables~\ref{tab:stability},\ref{tab:attraction} some classical results.

\section{Basic properties of the generalised \pe}

We start with the study of the eigenvalue problem~\eqref{ep-mixed} with mixed boundary 
condition.
We recall that $\O_r(y)$ denotes the connected component of $\O\cap B_r(y)$ 
containing the point $y\in\O$. 
Up to translation of the coordinate system, we can reduce to the case $y=0\in\O$,
calling for short $B_r:=B_r(0)$ and $\O_r:=\O_r(0)$.
The following properties are readily deduced using the fact that $\O$ is (path-)connected
and locally~smooth:
\Fi{Oapprox}
\forall r>0,\ \ \O_r=\bigcup_{0<\rho<r}\O_\rho,\qquad
\O=\bigcup_{\rho>0}\O_\rho,\qquad
\forall r>0,\ \exists R>0,\ \ \ol\O\cap B_r\subset\ol\O_R.
\Ff

The truncated domain $\O_r$ can be very irregular.
However \cite[Lemma~1]{Holes} ensures 
that the normals to $\O$ and $B_r$ are never parallel on the whole common boundary
$\partial B_r\cap\partial\O$ for a.e.~$r>0$, that~is,
\Fi{rSard}
	\text{for a.e.~}r>0,\quad
	\nu(x)\.\frac{x}{|x|}\neq\pm1\quad
	\text{for all }x\in \partial\O\cap\partial B_r.
\Ff	
This property is derived applying the Morse-Sard theorem~\cite{Morse-Sard}
to the distance function to $\partial\O$. This is where the $C^d$ regularity 
hypothesis on $\partial\O$ is required, which is inherited by the distance function,
	see~\cite[Lemma~14.16]{GT}.
It follows that $B_r\cap\O$ is a Lipschitz open set for a.e.~$r>0$, which in turn
allows us to invoke the solvability 
theory of~\cite{Lie86} (or~\cite{Stamp} in the self-adjoint case)
for the~problem
\Fi{mixed=f}
\begin{cases}
	-\L u=g(x) & \text{ in }\O_r\\
	\B u=0 & \text{ on } (\partial \O_r)\cap B_r \\
	u=0 & \text{ on }\partial \O_r\cap\partial B_r.
\end{cases}\Ff

\begin{lemma}\label{lem:mixed}
	Let $r>0$ be such that \eqref{rSard} holds, and 
	suppose that the problem~\eqref{mixed=f} with $g\equiv0$ admits a supersolution
	$\phi\in C^{2,\alpha}(\O_r)$ such that $\phi>0$ in $\ol \O_r$.
	Then, for any $g\in C^{0,\alpha}(\O_r)$, 
	the problem~\eqref{mixed=f} admits a unique solution 
	$u\in C^{2,\alpha}(\O_\rho)\cap C^0(\ol\O_r)$, for all $\rho<r$.
	
	Moreover, if $g\geq0$ (resp.~$\leq0$) then $u\geq0$ (resp.~$\leq0$),
	with strict inequality in $\ol\O_r\setminus\partial B_r$ if $g\not\equiv0$.
\end{lemma}

\begin{proof}
	We set $u=\phi w$, getting the following problem
	for~$w$:
	$$\begin{cases}
	-\t\L w=\frac g{\phi} & \text{ in }\O_r\\
	\beta\.\nabla w+\frac{\B \phi}{\phi}w=0 & 
	\text{ on } (\partial \O_{r})\cap B_{r} \\
	w=0 & \text{ on } \partial \O_{r}\cap \partial B_r,
	\end{cases}$$
	where
	$$\t\L w:=\nabla\.(A(x)\nabla w)+
	\left(b(x)+\frac{2A(x)\nabla\phi}{\phi}\right)\.\nabla w+
	\frac{\L\phi}{\phi}\,w.$$
	The zeroth order term of $\t\L$ is nonpositive because $\L\phi\geq0$, whereas
	that of the boundary operator on $(\partial \O_{r})\cap B_{r}$ is nonnegative 
	because $\B \phi\geq0$.
	Moreover the fact that $\nu(x)$ is not parallel to $x$ on 
	$\partial \O\cap\partial B_{r}$ yields the $\Sigma$-wedge condition of~\cite{Lie86}.
	We are thus in the framework of \cite[Theorem~1, Lemma 1]{Lie86}
	(see the comment after the proof of that theorem 
	for the case where $\B \phi\equiv0$ on $(\partial \O_{r})\cap B_{r}$), 
	which provides us with a unique
	solution $w\in C^{2,\alpha}(\O_\rho)\cap C^0(\ol\O_{r})$, for all~$\rho<r$.
	This gives the desired solution $u$ to~\eqref{mixed=f}.
	
	The last statement is readily derived applying the weak \MP\
	and Hopf's lemma on $(\partial \O_{r})\cap B_{r}$ (which is regular)
	to the function~$w$.	
\end{proof}

The final argument of the proof of Lemma~\ref{lem:mixed}
does not require the Lipschitz-regularity 
of $\O_r$, nor the $C^{2,\alpha}$ regularity of $\phi$. We indeed have the following.
\begin{lemma}\label{lem:monotone}
	Assume that~\eqref{ep-mixed} admits a supersolution
	$\phi\in W^{2,\infty}(\O_r(y))$ such that $\phi>0$ in $\ol \O_r(y)$.
	Then any subsolution $u$ of~\eqref{ep-mixed} satisfies $u\leq0$ in $\ol \O_r(y)$.
%
%
%
%
\end{lemma}


We now derive the existence of the \pe\ for a set of radii $r$ which is smaller than
the one provided by~\eqref{rSard}, but still exhausting $\R_+$ up to a zero measure set.
One possible strategy would be to apply the  
Krein-Rutman theory in the weak form of~\cite{Nus81} (see also~\cite{Bir95}), where the
strong positivity of the inverse operator is substantially relaxed.
It however requires the compactness of such operator, hence some global regularity 
estimates 
%
which are available in the literature only under some involved geometric hypotheses,
see~\cite{Lie89}.
We use a different, more direct approach, used in~\cite{BD06} in the context of fully nonlinear operators,
which requires some weaker estimates that we obtain here with the aid of barriers from \cite{Lie86}.

\begin{theorem}\label{thm:mixed}
	Let $y\in\O$. For a.e.~$r>0$, there exists a unique $\lambda\in\R$ for which
	the eigenvalue problem~\eqref{ep-mixed} 
	admits a classical solution $\vp$ which is positive in~$\ol\O_r\setminus\partial B_r$.
\end{theorem}

\begin{proof}
	We can assume without loss of generality that $y=0$.
	Let $\mc{R}\subset\R_+$ be the set of $r$ for which~\eqref{rSard} holds; we know that
	$\R_+\setminus\mc{R}$ has zero Lebesgue measure. We shall derive the 
	existence of the \pe\ $\lambda(r)$ 
	for any $r\in\mc{R}$. 
%
%
	We shall further characterise it through
	the ``generalised" formulation
$$\lambda(r):=
\sup\{\lambda\ :\ \exists\phi>0,\ (\L+\lambda)\phi\leq0 \text{ in }
\O_r,\ \B\phi\geq0\text{ on }(\partial \O_r)\cap B_r \}.
$$
The functions $\phi$ in the above definition are understood to belong to
$C^{2,\alpha}(\O_\rho)$ for any $\rho<r$ 
(recall that $\bigcup_{0<\rho<r} \O_\rho(y)=\O_r$).
We emphasise that no condition on the truncated boundary
$\partial \O_r\cap\partial B_r$ is imposed, 
because the positivity of~$\phi$ already implies that it
is a supersolution with respect to the Dirichlet boundary condition.

	By the definition of $\lambda(r)$, there exists an increasing sequence $\seq{\lambda}$ 
	converging to
	$\lambda(r)$, with an associated sequence $\seq{\phi}$ of positive functions
	in $C^{2,\alpha}(\O_\rho)$ for any $\rho<r$, satisfying 
	$$(\L+\lambda_n)\phi_n\leq0 \;\text{ in }
	\O_r,\qquad \B\phi\geq0\;\text{ on }(\partial \O_r)\cap B_r.$$
	Next, take an increasing sequence $\seq{r}$ in $\mc{R}$ 
	converging to $r$.
	For $n\in\N$, we consider the problem
	$$\begin{cases}
	-(\L+\lambda_n)u=1 & \text{ in }\O_{r_n}\\
	\B u=0 & \text{ on } (\partial \O_{r_n})\cap B_{r_n}\\
	u=0 & \text{ on }\partial \O_{r_n}\cap \partial B_{r_n}.
	\end{cases}$$
	Observe that $\phi_n>0$ on $\ol\O_{r_n}$ by Hopf's lemma.
	Then, we can apply Lemma~\ref{lem:mixed} and infer the existence of a unique
	classical solution $u=u_n$ of the above problem, which is positive in 
	$\ol\O_{r_n}\setminus\partial B_{r_n}$.
	Suppose for a moment that (up to subsequences) 
	$(\max_{\ol\O_{r_n}}u_n)_{n\in\N}$ is bounded.
	Then, by \cite[Lemma 1]{Lie86} (a subsequence of) the sequence $\seq{u}$ converges
	in $C^{2,\beta}(\O_\rho)$, for any $\rho<r$ and fixed $\beta<\alpha$, to a nonnegative 
	solution~$\phi$~of
	$$-(\L+\lambda(r))\phi=1 \;\text{ in }
	\O_r,\qquad \B\phi=0\;\text{ on }(\partial \O_r)\cap B_r.$$
	The function $\phi$ actually belongs to $C^{2,\alpha}(\O_\rho)$ for any $\rho<r$,
	always by \cite[Lemma 1]{Lie86}.
	Furthermore, $\phi>0$ in $\O_r$ due to the \SMP, and we have that
	$-(\L+\lambda(r))\phi=1\geq\e\phi$ in $\O_r$, for some $\e>0$,
	contradicting the definition of $\lambda(r)$.
	Therefore, the sequence $(\max_{\ol\O_{r_n}}u_n)_{n\in\N}$ necessarily diverges.
	
	Define 
	$$v_n:=\frac{u_n}{\max_{\ol\O_{r_n}}u_n}.$$
	The sequence $\seq{v}$ converges locally uniformly (up to subsequences) 
	to a solution $0\leq\vp\leq1$ of 
	$$-(\L+\lambda(r))\vp=0 \;\text{ in }
	\O_r,\qquad \mc{B}\vp=0\;\text{ on }(\partial \O_r)\cap B_r.$$
	Furthermore, using a standard barrier, one infers that $\vp$ can be 
	continuously extended to~$0$ at all regular points of 
	$\partial \O_{r}\cap\partial B_{r}$.
	In order to handle the ``wedges''
	$\partial \O_{r}\cap\partial B_{r}\cap\partial\O$, we make use
	of the barrier provided by \cite[Lemma~2]{Lie86}.
	Similarly to the barrier of~\cite{Miller}, 
	this is a function $w$ which vanishes at a given point
	$\xi\in\partial \O_{r}\cap\partial B_{r}\cap\partial\O$,
	it is positive in 
	$\ol{\O_{r}\cap B_\rho(\xi)}\setminus\{\xi\}$ for some $\rho>0$, and satisfies
	$$-(\L+\lambda(r))w\geq1\;\text{ in }
	\O_r\cap B_\rho(\xi),$$
	but it additionally fulfils 
	$$\beta\.\nabla w\geq1 \;\text{ on }
	(\partial \O_r)\cap B_r\cap B_\rho(\xi).$$
	Take a large enough constant $k$ so that
	\Fi{kappa}
	kw\geq 1\ \text{on }\O_r\cap\partial B_\rho(\xi),\qquad
	k>\sup_\O c,\qquad k>-\inf\gamma.
	\Ff
	For $n\in\N$, let $x_n$ a point where $v_n-kw$ attains its maximum on 
	$\ol{\O_{r_n}\cap B_\rho(\xi)}$ and assume by contradiction that $(v_n-kw)(x_n)>0$.
	Observe that $x_n$ can neither belong to $\partial B_{r_n}$, because $v_n=0$ there, 
	nor to $\O_r\cap\partial B_\rho(\xi)$, where $kw\geq 1\geq v_n$.
	If $x_n\in \O_{r_n}\cap B_\rho(\xi)$
	then at such point there holds $D^2 v_n\leq D^2 kw$ (as matrices) and $\nabla v_n=\nabla kw$, 
	from which we derive
	$$-\frac1{\max_{\ol\O_{r_n}}u_n}+k\leq (\L+\lambda_n)(v_n-kw)(x_n)
	\leq c(x_n)(v_n-kw)(x_n)\leq\sup_\O c.$$
	The above left-hand side converges as $n\to\infty$ to $k$, which is larger than $\sup_\O c$
	by~\eqref{kappa}, hence 
	this case is ruled out for~$n$ sufficiently large.
	Therefore, the 
	only remaining possibility for $n$ large is that $x_n\in\partial\O$.
	In such case, recalling that $\beta$ points outside~$\O$, we find that
	$\beta\.\nabla(v_n-kw)(x_n)\geq0$, whence
	$$\B v_n(x_n)\geq k\beta\.\nabla w(x_n)+\gamma v_n(x_n)\geq k+\min\{\inf \gamma,0\},$$
	which is positive by~\eqref{kappa}, contradicting the boundary condition for $v_n$.
	This shows that $v_n\leq kw$ on $\ol{\O_{r_n}\cap B_\rho(\xi)}$ for $n$ sufficiently large
	and thus the function
	$\vp$ can be continuously extended to $0$ at $\xi$.
	
	Another consequence of the comparison with the barriers at every point of 
	$\partial \O_{r}\cap\partial B_{r}$ is that any sequence of points $\seq{y}$ for which
	$v_n(y_n)=1=\max v_n$, does not approach $\partial B_r$ as $n\to\infty$.
	This yields $\max\vp=1$ and thus $\vp>0$ in $\ol\O_r\setminus\partial B_r$ by the 		
	\SMP\ and Hopf's lemma. The function $\vp$ is thus a positive solution 
	to~\eqref{ep-mixed} with eigenvalue~$\lambda$. 
	
	To complete the proof, it remains to show the uniqueness of the \pe.
	We are not able to do it for every $r\in\mc{R}$. However, using 
	Lemma~\ref{lem:monotone}, 	we shall derive it a.e.~in $\mc{R}$.
	Observe that~$\lambda(r)$  is nonincreasing with respect to $r$ by its very definition.
	There exists then a subset $\t{\mc{R}}$ of $\mc{R}$ such that
	$\mc{R}\setminus\t{\mc{R}}$ has zero Lebesgue measure and 
	on which~$r\mapsto\lambda(r)$ is continuous, i.e., such that
	$$\forall r\in\t{\mc{R}},\quad
	\lambda(r)=\inf_{\mc{R}\ni \rho<r}\lambda(\rho)=\sup_{\mc{R}\ni \rho>r}\lambda(\rho).$$
	For fixed $r\in\t{\mc{R}}$, let $\lambda$ be an eigenvalue for~\eqref{ep-mixed} 
	with a positive eigenfunction. 
	Lemma~\ref{lem:monotone} implies that any eigenvalue for the mixed problem in 
	$\O_\rho$ with $\rho<r$ (resp.~$\rho>r$) 
	admitting a positive eigenfunction must be strictly larger (resp.~smaller)
	than $\lambda$.
	Therefore, the existence of the positive eigenfunctions associated with the
	$\l(\rho)$, $\rho\in\mc{R}$, proved above yields
	$$\forall\rho,\rho'\in\mc{R}\;\text{ with }\rho'<r<\rho,\quad
	\l(\rho')>\l>\l(\rho).$$
	We eventually deduce from the definition of $\t{\mc{R}}$ that $\lambda=\l(r)$.	
\end{proof}

In the sequel, the eigenvalue provided by Theorem~\ref{thm:mixed} will be 
called the \pe\ for~\eqref{ep-mixed} and will be denoted by $\lambda(y,r)$.
The associated positive solution will be called {\em a}
\pf\ for~\eqref{ep-mixed}. The simplicity of $\lambda(y,r)$ remains an open question.

We now show the convergence of $\l(y,r)$ towards~$\lB$.
\begin{proof}[Proof of \thm{lrinfty}]
Assume that $y=0$. We restrict ourselves to the values of~$r$ for which~\thm{mixed} applies.
This provides, for any of such values, a unique \pe~$\lambda(y,r)$ and a
principal eigenfunction~$\vp_r$ for the mixed 
problem~\eqref{ep-mixed}. We normalise the~$\vp_r$ by~$\vp_r(0)=1$.
In addition, Lemma~\ref{lem:monotone} implies that $\lambda(y,r)$ is
strictly decreasing with respect to $r$.
Next, by the definition of $\lB$, for any $\l<\lB$ there exists a positive supersolution $\phi$ 
of~\eqref{epB}. Applying again Lemma~\ref{lem:monotone} we get
$\l<\lambda(y,r)$ for (almost) every $r>0$, that is, $\lambda(y,r)\geq\lB$.
We can then define
$$\lambda:=\lim_{r\to\infty}\l(y,r)\geq\lB.$$
		
We now invoke the boundary Harnack inequality for the oblique derivative problem, which implies 
that for any $\rho>0$, there exists $C_\rho>0$ such that
$$\forall r>\rho+1,\quad
\sup_{\O_\rho}\vp_r\leq C_\rho\inf_{\O_\rho}\vp_r,$$
which in turn is smaller than $C_\rho\vp_r(0)=C_\rho$.
The above Harnack inequality is derived using a chaining argument in which one applies
the interior Harnack inequality on balls contained in $\O$ and the 
boundary Harnack inequality for the oblique derivative problem on balls intersecting $\partial\O$.
The latter follows combining the 
local maximum principle of \cite[Theorem~3.3]{Lie87}
with the weak Harnack inequality of \cite[Theorem~4.2]{Lie01}.
We emphasise that, unlike for Dirichlet problems, it is valid up to the boundary of~$\O$, 
as a consequence of the strict positivity of solutions
of the oblique derivative problem in the whole $\ol\O$.
The Harnack inequality, 
together with the local boundary estimate of \cite[Lemma~1]{Lie86}
and the last property in~\eqref{Oapprox}, 
imply that
(a subsequence of) the~$\vp_r$ converges in $C^2_{loc}(\ol\O)$ to a nonnegative
solution $\vp$ of the eigenvalue problem~\eqref{epB}. Because $\vp$ satisfies $\vp(0)=1$, 
the \SMP\ yields $\vp>0$ in $\O$. We deduce 
at the same time that $\lambda=\lB$, which concludes the proof of the theorem, as well as
the existence of a generalised \pf\ associated with~$\lB$.
Namely, we have proved~\thm{ef} in the case~$\lambda=\lB$.
\end{proof}

We can now derive~\thm{ef} from 
\thm{lrinfty} following the same lines as  in the proof 
of~\cite[Theorem 1.4]{BR4}. 

\begin{proof}[Proof of \thm{ef}]
We have already derived the result for~$\lambda=\lB$.
Take $\lambda<\lB$. Assume that $0\in\O$.
Consider an increasing, diverging sequence $\seq{r}$ 
for which the conclusions of Theorems~\ref{thm:lrinfty}, \ref{thm:mixed}
hold with $y=0$.
Since $\O$ is unbounded and connected,
$\O_{r_n}\backslash\ol B_{r_n-1}\neq\emptyset$ for all $n\in\N$. Let
$(g_n)_{n\in\N}$ be a family of smooth, nonpositive and not identically
equal to zero functions such that
$$\forall n\in\N,\quad \supp g_n\subset\O_{r_n}\backslash\ol B_{r_n-1}.$$
Consider the problems
	$$\begin{cases}
	-(\L+\l)u=g_n & \text{ in }\O_{r_n}\\
	\B u=0 & \text{ on } (\partial \O_{r_n})\cap B_{r_n}\\
	u=0 & \text{ on }\partial \O_{r_n}\cap\partial B_{r_n}.
	\end{cases}$$
The \pf\ associated with $\lambda(y,r_{n+1})$ is a supersolution for this problem with
$0$ in place of $g_n$, because
$\lambda(y,r_{n+1})>\lB>\lambda$ thanks to \thm{lrinfty}, and it is positive in~$\ol\O_{r_n}$. 
Hence, Lemma~\ref{lem:mixed} provides us with a positive
bounded solution $u=u_n$ of the problem with $f_n$.
We then define the sequence $\seq{v}$ by
$$v_n(x):=\frac{u_n(x)}{u_n(0)}.$$
Using the Harnack inequality and the local boundary estimates, 
exactly as in the proof of \thm{lrinfty},
we can extract a subsequence converging locally uniformly in~$\ol\O$
to a nonnegative solution $\vp$
of~\eqref{epB}. Moreover, $\vp(0)=1$ and therefore $\vp>0$ by the \SMP.
\end{proof}


\section{Generalised principal eigenvalue and stability}

In this section we study the link between the stability of the null state for~\eqref{BP} and
the different notions of generalised \pe\ defined in Sections~\ref{intro:lB},~\ref{intro:stability}.
Throughout this section, these notions are applied to the linearised operator around~$0$:
$$\L w:=\Tr(A(x)D^2 w)+b(x)\.\nabla w+f_s(x,0)w.$$
\begin{proposition}\label{pro:stability}
	The following properties hold:
	\begin{enumerate}[$(i)$]
	
	\item if $\lB^b>0$ then~$0$ is uniformly stable w.r.t.~compact
	perturbations; 
	
	\item if $\l_{\mc{B}}^{p,b}>0$ then~$0$ is uniformly stable;
	
	\item if~$0$ is locally stable then $\mu_{\mc{B}}^b\geq0$.
	
	\end{enumerate}
\end{proposition}

\begin{proof}
	$(i)$\\
	Assume that $\lB^b>0$, namely, there exists $\lambda>0$ and a function $\phi$
	satisfying
	$$\phi>0 \inn\O,\qquad
	\sup_\O\phi<+\infty,\qquad(\L+\lambda)\phi\leq0 \inn\O,
	\qquad\B\phi\geq0\on\partial \O.$$
	For $\e>0$, we call $v^\e(t,x):=\e \phi(x) e^{-\frac\lambda2 t}$. This function satisfies
	$\B v^\e\geq0$ on~$\R_+\times\partial\O$, while, in $\R_+\times\O$,
	\begin{align*}
	\partial_t v^\e -
	\dv (A(x) \nabla v^\e)-b(x)\.\nabla v^\e-f(x,v^\e)
	&\geq \frac\lambda2 v^\e+f_s(x,0)v^\e-f(x,v^\e)\\
	&=\Big(\frac\lambda2 +f_s(x,0)-f_s(x,s(t,x)v^\e)\Big)v^\e,
	\end{align*}
	for some $s(t,x)\in(0,1)$. 
	It then follows from the boundedness of $\phi$ and
	the regularity assumption on $f$, that~$v^\e$ is a supersolution 
	to~\eqref{BP} provided $\e$ is sufficiently small.
	We choose $\e$ so that this is the case, then, for given $R>0$, we set 
	$$\delta:=\e\min_{\ol\O\cap\ol B_R}\phi,$$
	which is positive because $\phi>0$ on $\ol\O$ by Hopf's lemma.
	Thus, by the comparison principle,
	any solution~$u$ to~\eqref{BP} with an initial datum
	$u_0$ supported in $\ol\O\cap\ol B_R$ and satisfying
	$\|u_0\|_\infty\leq\delta$, is bounded from above
	by $v_\e$. Likewise,  
	$-v^\e$ is a subsolution to~\eqref{BP} and thus $u\geq-v^\e$. As a consequence,
	$u$ converges uniformly to $0$ as $t\to+\infty$.
	
	\medskip
	$(ii)$\\
	Test functions for $\lB^{p,b}$ differ from those for $\lB^b>0$ by the additional requirement
	$\inf_\O\phi>0$. Thus, proceeding as in the case of~$(i)$ we infer that the uniform convergence
	to $0$ as $t\to+\infty$ holds true for any initial datum $u_0$ satisfying
	$\|u_0\|_\infty\leq\e\inf_\O\phi$.
%
	
	\medskip
	$(iii)$\\
	Assume that $\mu_{\mc{B}}^b<0$. Then, there exists $\lambda<0$ and
	a positive function $\phi$ such that
	$$\sup_\O\phi<+\infty,\qquad\L\phi\geq\lambda\phi \inn\O,
	\qquad\B\phi\leq0\on\partial \O.$$
	For $\e>0$, the function $\e\phi$ satisfies
	$$-\dv \big(A(x) \nabla(\e\phi)\big)-b(x)\.\nabla (\e\phi)
	\leq \big(f_s(x,0)-\lambda\big)\e\phi,$$
	which, by the regularity of $f$,
	is smaller than $f(x,\e\phi)$ provided $\e$ is sufficiently small.
	Hence, for such values of $\e$, $\e\phi$ is a stationary subsolution 
	to~\eqref{BP} and therefore, by comparison, the solution having $\e\phi$ as initial datum
	remains larger than $\e\phi$ for all $t>0$ (it may also blow up in finite time). 
	Since $\phi$ is bounded, we deduce that $0$ is not locally stable.	 
\end{proof}

Next, we investigate the repulsion property, see Definition~\ref{def:repulsive}.
\begin{proposition}\label{pro:locrepulsive}
	If $\lB<0$ then $0$ is locally repulsive.
\end{proposition}
	
\begin{proof}
	By \thm{lrinfty}, for any fixed $y\in\O$,
	there exists $r$ large enough for which $\lambda(y,r)<0$. 
	Up to translation, we can take $y=0$.
	Let $\vp_r$ be the \pf\ associated with $\lambda(0,r)$.
	The same arguments as in the proof of Proposition~\ref{pro:stability} show that 
	$\e\vp_r$ is a subsolution to the first equation of~\eqref{BP} 
	in $\O_r $ provided $\e>0$ is sufficiently small.
	It further satisfies $\B=0$ on $(\partial \O_r)\cap B_r$ and
	it vanishes on $\partial \O_r\cap\partial B_r $. 
	Consider a solution $u$ to~\eqref{BP} with an arbitrary initial datum
	$u_0\geq0,\not\equiv0$.
	The parabolic \SMP\ and Hopf's lemma yield $u(t,x)>0$ for all $t>0$,
	$x\in\ol\O$.
	We can therefore find $\e$ small enough, depending on $r$,
	such that $\e\vp_r< u(1,\cdot)$ 
	in $\O_r $. Applying the parabolic comparison principle in this set
	we deduce that the inequality $u(t,\cdot)>\e\vp_r$ holds true in $\O_r $
	for all $t>1$. As a consequence, for $x\in\O_r$, we find that
	$$u_\infty(x):=\liminf_{t\to+\infty}u(t,x)\geq\e\vp_r(x)>0.$$
	Owing to the second property in~\eqref{Oapprox}, this proves $0$ is locally repulsive.
\end{proof}

We point out that Propositions~\ref{pro:stability} and~\ref{pro:locrepulsive}, 
with a minor modification in the proof of the latter and in the first statement of the former,
hold true in the Dirichlet case $\B u=u$ 
(relaxing the local repulsion by $u_\infty>\bar u$
in $\O$).
	
\begin{proof}[Proof of Table~\ref{tab:stability}]
	All implications in the table follow from Propositions~\ref{pro:stability} and~\ref{pro:locrepulsive},
	with the exception of the second implications of the first and third line.
	These implications directly follow from the others by noticing that 
	the local stability with respect to compact perturbations prevents 
	the local repulsion.
\end{proof}

%
%
	
\begin{proposition}\label{pro:ce-stability}
	The following hold:
	\begin{enumerate}[$(i)$]
		\item $(\lB\geq)\;\lB^p>0$ does not imply the local stability  w.r.t.~compact perturbations;
		\item $\lB^p>0$ does not imply the uniform stability  w.r.t.~compact 
			perturbations, even when $f$ is linear or $f_{ss}<0$;
		\item $\lB^b>0$ does not imply the local stability, even when $f$ is linear or $f_{ss}<0$;
		\item the local stability  w.r.t.~compact perturbations does not imply $\mu_\B^b\geq0$;
		\item neither $\lB^b<0$ nor $\lB^p<0$ imply the local repulsion.
	\end{enumerate}
\end{proposition}

\begin{proof}
	We start with $(ii)$.\\
	Consider the linear operator 
	$$
	\L w:=
	w''-2w'+c(x)w\ \ \inn\R,
	$$
	with $c$ nondecreasing and satisfying
	$$c(x):=\begin{cases}
	-\frac12 &\text{if }x\leq0\\
	\frac12 &\text{if }x\geq1.
	\end{cases}$$
	Explicit computation shows that the function $\phi(x):=e^x+\frac14$ satisfies
	$\L\phi\leq-\frac14\phi$ in~$\R$, whence $\lB^p\geq1/4$.
	Let $u$ be a solution of $\partial_t u=\L u$ for $t>0$, $x\in\R$,
	with a compactly supported initial datum $u_0\geq0,\not\equiv0$.
	The function $\t u(t,x):=u(t,x+2t)$ satisfies
	$$
	\partial_t \t u=\partial_{xx}\t u+\frac12\t u,\quad t>0,\ x>-2t+1.
	$$
	It is clear that such solution $\t u$ tends to $+\infty$ as $t\to+\infty$,
	locally uniformly in space.	
	If we instead considered the $KPP$
	equation $\partial_t u=\L u-\frac12u^2$, 
	for which $f_{ss}=-1$ and~$\L$ is the linearised operator, 
	we would find that $\t u\to1$ as $t\to+\infty$ locally uniformly in space.
	In both cases, $u$ does not converge uniformly to $0$.

$(i)$\\
	We seek for a counter-example in the form
	\Fi{evolfc}
	\partial_t u=\partial_{xx}u-2\partial_{x}u+f(u)+\Big(c(x)+\frac12\Big)u,\quad t>0,\ x\in\R,
	\Ff
	where $c$ is as in the proof of~$(ii)$.
	As for the nonlinearity $f:[0,1]\to\R$, we take it of the bistable type, in 
	the sense of~\cite{AW}, satisfying 
	$$\int_0^1 f>0,\qquad f'(0)=-\frac12,\qquad f(s)>-\frac s2\ \text{ for }s>0.$$
	It follows from~\cite{AW} that the
	asymptotic speed of spreading for the equation
	\Fi{RDv}
	\partial_t v=\partial_{xx}v+f(v),\quad t>0,\ x\in\R,
	\Ff  
	is positive, and, up to increasing $f$, we can assume that it is larger than~$3$.
	Namely, there exists $R$ sufficiently large such that any solution
	emerging from an initial datum $v_0\geq\1_{B_R}$ satisfies
	$$\lim_{t\to+\infty}\Big(\min_{|x|\leq 3t}v(t,x)\Big)=1.$$
	The linearised operator around the null state for~\eqref{evolfc} is the same $\L$ as
	in the proof of~$(ii)$, for which we have seen that $\lB^p\geq1/4$.
	Let $u$ be a solution to~\eqref{evolfc} with a compactly supported initial datum $u_0\geq0,\not\equiv0$.
	Because $f(s)\geq-s/2$, this is a supersolution of $\partial_t u=\L u$ and therefore, as seen before,
	the function $\t u(t,x):=u(t,x+2t)$ tends to 
	$+\infty$ as $t\to+\infty$ locally uniformly in $x$.
	On the other hand, $\t u$ is a supersolution to~\eqref{RDv}, whence
	the comparison with a solution $v$ with initial datum~$\1_{B_R}$ yields
	$$\liminf_{t\to+\infty}\Big(\min_{|x|\leq 3t}\t u(t,x)\Big)\geq1.$$
	From this we eventually derive
	$$\forall x\in\R,\quad
	\liminf_{t\to+\infty}u(t,x)\geq1.$$
	This means that $0$ is not locally stable.
	
	$(iii)$\\
	Consider the operator
	$$\L u=u''+k u'+c(x)u,$$
	with $k$ to be chosen and $c$ smooth and satisfying
	$$c(x):=\begin{cases}
	-1 &\text{if }x<-1\\
	1 &\text{if }x\geq0.
	\end{cases}$$
	Direct computation shows that the decreasing
	function $\phi(x):=(1+e^x)^{-1}$ satisfies 
	$$\phi''<0\for{x<0},\qquad
	\phi''<-\phi'\for{x>0},\qquad
	\phi<2\phi'\for{x>0}.$$
	It follows that, for $k>0$ large enough, $(\L+1)\phi\leq0$ in $\R$, whence 
	$\lB^b\geq1$.
	On the other hand, the increasing function $\psi(x):=(1+e^{-x})^{-1}$ satisfies
	$$\psi''>0\for{x<0},\qquad
	\psi<2\psi'\for{x<0},\qquad
	-\psi''<\psi'\for{x>0}.$$
	Therefore, for $k>0$ large enough, $(\L-1)\psi\leq0$ in $\R$. This means that
	$\mu_\B^b\leq-1$ and thus $0$ is not locally stable for the problem
	$\partial_t u=\L u$ in $\R$ owing to Table~\ref{tab:stability}.
	The same is true for the equation $\partial_t u=\L u-u^2$.
%
	
	$(iv)$\\
	A counter-example is given by the equation 
	$$
	\partial_t u=\partial_{xx}u+3\partial_{x}u+u(1-u),\quad t>0,\ x\in\R,
	$$
	which is simply the Fisher-KPP equation in the moving frame with speed $4$.
	Namely, the solution is $u(t,x)=v(t,x+3t)$ with $v$ satisfying the standard 
	Fisher-KPP equation. According to~\cite{KPP,AW}, the spreading speed for the latter is $2$,
	which implies that if the initial datum for $u$ (which is the same as for $v$)
	is nonnegative and compactly supported then
	$u(t,x)=v(t,x+3t)\to0$ as $t\to+\infty$, locally uniformly in $x$. 
	Hence~$0$ is locally stable with respect to compact perturbations.
	On the other hand, the linearised operator $\L u=u''+4u'+u$ in $\R$
	satisfies $\mu_\B^b\leq-1$, as it is immediately seen
	taking $\phi\equiv1$ in the definition.
	
	$(v)$\\
	The counter example is the same as in $(iv)$.
	We have seen that the null state is locally stable with respect to compact perturbations,
	whence it is not locally repulsive.
	Moreover, one can show that the linearised operator $\L$
	satisfies  $\lB^p=-1$. This can be deduced from \thm{relations}$(iv)$
	and then taking $\phi=1$ in the definitions of $\lB^p$ and $\mu_\B^b$.
	In order to estimate $\lB^b$, we consider a positive function $\phi$ such that
	$(\L+\lambda)\phi\leq0$ in $\R$ for some $\lambda\in\R$. 
	Then $\t\phi(x):=\phi(x)e^{2x}$ is positive and satisfies
	$\t\phi''\leq(3-\lambda)\t\phi$. This clearly implies that $\lambda\leq3$
	and moreover, as a consequence for instance of the Landis conjecture
	(which holds in this case, see \cite[Theorem~1.4]{Landis}),
	that $\t\phi(x)e^{\kappa|x|}\to+\infty$ as $|x|\to\infty$,
	for every $\kappa>\sqrt{3-\lambda}$. It follows that $\phi(x)\to+\infty$ as $x\to-\infty$
	if $\lambda>-1$, whence $\lB^b\leq-1$.
	\end{proof}

%

We pass now to the sufficient condition for the uniform repulsion.
\begin{proof}[Proof of \thm{LambdaB}]
	Let $u$ be a solution to~\eqref{BP}
	with a nonnegative initial datum $u_0\not\equiv0$.
	First of all, we reduce to the case where $u$ is bounded by changing the initial datum 
	into $\max\{u_0,1\}$ and the nonlinearity into $f(x,s)-k s^2$. Indeed, in such case,
	for $k$ sufficiently large, 
	the corresponding solution $\t u$ satisfies $0\leq \t u\leq\min\{1,u\}$ thanks to the comparison principle,
	hence a lower bound for $\t u$ entails the same bound for~$u$. 
	We then assume without loss of generality that $u$ is bounded.
	The function 
	$$u_\infty(x):=\liminf_{t\to+\infty}u(t,x)$$
	is well defined for all $x\in\ol\O$.
	
	We proceed in three steps: 
	we first show that $u_\infty$ is larger than some positive constant on a relatively dense set, 
	next, we derive a simple geometrical property about coverings of bounded sets which allows us,
	in the last step,
	to extend the lower bound for $u$ to the whole $\O$ using the Harnack inequality
	and a chaining argument.
	
	\step{1} {\em Lower bound on a relatively dense set.}\\
	By hypothesis, $\LB<0$. Take $\LB<\lambda<0$. The definition of $\LB$ and \thm{lrinfty} provide us with a radius
	$R>1$ such that, for any $y\in\O$, there exists $R-1<r_y<R$ for which 
	the problem~\eqref{ep-mixed} admits a principal eigenvalue satisfying
	$\l(y,r_y)\leq \lambda$.
	Let us call $\varphi^y$ the principal eigenfunction associated with~$\l(y,r_y)$,
	normalised by $\|\varphi^y\|_{L^\infty(\O_{r_y}(y))}=1$. Next, because $f(x,0)=0$ and
	$s\mapsto f(x,s)$ is of class~$C^1$ in some interval $[0,\e]$,
	uniformly with respect to~$x$, there holds that
	$$\forall x\in\O,\ s\in(0,\e],\quad f(x,s)\geq \bigg(f_s(x,0)+\frac\l2\Bigg)s.$$
	Consider a smooth function $\sigma:\R\to\R$ which is increasing and satisfies 
	$$\sigma(-\infty)=-\infty,\qquad
	\sigma(+\infty)=0,\qquad
	\sigma'\leq-\frac\lambda2\inn\R.$$
	Then define $v^y(t,x):=\e\vp^y(x)e^{\sigma(t)}$. This function 
	satisfies
	$$\partial_t v^y-\dv \big(A(x) \nabla v^y\big)-b(x)\.\nabla v^y\leq
	\bigg(f_s(x,0)+\frac\l2\Bigg)v^y\leq f(x,v^y)\inn\R\times\O_{r_y}(y),$$
	together with 
	$$\B v^y=0\on \R\times\big((\partial \O_{r_y}(y))\cap B_{r_y}(y) \big),\qquad 
	v^y=0\on\R\times\big(\partial \O_{r_y}(y)\cap\partial B_{r_y}(y)\big).$$
	On the other hand, $u$ is a supersolution of this mixed problem.
	Moreover, because $u(t,x)>0$ for all $t>0$ and $x\in\ol\O$, as a consequence of 
	the parabolic strong maximum principle and Hopf's lemma, for any $y>0$, 
	we can find $\tau_y\in\R$ such that $v^y(\tau_y,x)\leq u(1,x)$ for all
	$x\in\ol\O_{r_y}(y)$. 
	We can therefore apply the comparison principle in this set and deduce in particular that
	$$\forall x\in\O_{r_y}(y),\quad
	u_\infty(x)\geq \lim_{t\to+\infty}v^y(t,x)=\e\vp^y(x).$$
	Recalling that $r_y<R$ and that $\sup_{\O_{r_y}(y)}\vp^y=1$, we eventually derive
	\Fi{lbreldense}
	\forall y\in\O,\quad
	\sup_{\O_R(y)} u_\infty\geq\e.
	\Ff
	
	\step{2} {\em The $r$-{\em internal covering number} of a set $E\subset B_R\subset \R^d$ 
		is controlled by a constant only depending on $r,R,d$.}\\
	The $r$-internal covering number of $E$ is the minimum cardinality of $Q\subset E$
	such that $E\subset \bigcup_{a\in Q}B_r(a)$.
	Then we need to show that there exists a set of (not necessarily distinct) points 
	$\{x_1,\dots,x_n\}\subset E$, with $n$ only depending on $r,R,d$, such that 
	$$E\subset \bigcup_{j=1}^{n} B_r(x_j).$$
	Consider the set of points
	$$\mc{Z}:=\Big(\frac r{2\sqrt{d}}\Z\Big)^d\cap B_R.$$
	It follows that 
	$$B_R\subset \bigcup_{z\in\mc{Z}} B_{\frac r2}(z).$$
	For $z\in\mc{Z}$ such that $B_{\frac r2}(z)\cap E\neq\emptyset$, we choose a point in 
	$B_{\frac r2}(z)\cap E$
	and we call it~$x(z)$. Then the family of balls 
	$$\{B_r(x(z))\ :\ z\in\mc{Z},\ B_{\frac r2}(z)\cap E\neq\emptyset\}$$
	is a covering of $E$. This step is proved, because the number of elements of $\mc{Z}$ 
	only depends on $r,R,d$.
		
	\step{3} {\em Uniform lower bound.}\\
	Let $\rho$ be the 
	radius provided by the uniform $C^2$ regularity of~$\partial\O$ (see Section~\ref{sec:main})
	and call $r:=\rho/3$.
	Consider then $y\in\O$.
	Applying step 2 with $E=\O_R(y)$ we find a set $\{x_1,\dots,x_n\}\subset \O_R(y)$, 
	with $n$ depending on $r,R,d$, such that 
	$$\O_R(y)\subset \bigcup_{j=1}^{n} B_r(x_j).$$
	Up to a permutation, we can reduce to the case where 
	\Fi{supOR}
	\sup_{B_r(x_1)\cap\O}u_\infty\geq\sup_{\O_R(y)}u_\infty,
	\Ff
	and, in addition, the set
	\Fi{connected}
	\bigcup_{j=1}^{i} \big(B_r(x_j)\cap\O\big)
	\Ff
	is connected for every $i=1,\dots,n$.
	The latter property 
	is deduced from the fact that~$\O_R(y)$ is connected.
	We now apply the parabolic Harnack 
	inequality, 
	which is possible because
	$f(x,u)$ can be written as $\t c(x)u$ with $\t c\in L^\infty(\O)$ (recall that
	$u$ is bounded and $s\mapsto f(x,s)$ is locally Lipschitz continuous, uniformly in $x$).
	For the values of $j$ such that $B_{2r}(x_j)\subset\O$, we apply the interior Harnack 
	inequality (see, e.g., \cite[Corollary 7.42]{Lie}), which gives us
	a constant~$C\in(0,1)$, only depending on $r$ as well as the ellipticity constants and the $L^\infty$
	norms of the coefficients, such that
	$$\forall t>0,\quad
	\inf_{x\in B_r(x_j)}u(t,x)\geq C\sup_{x\in B_r(x_j)}u(t+1,x).$$
	The same 
	conclusion holds true restricted to $B_r(x_j)\cap\O$ when
	$B_{2r}(x_j)$ intersects $\partial\O$, thanks to the boundary Harnack inequality
	for the oblique derivative problem, c.f., 
	\cite[Theorem~7.5]{Lie01} and \cite[Theorem~7.48]{Lie}, 
	with $C$ also depending on $\rho,\|\Psi\|_{C^2}$ provided by the
	uniform regularity of $\partial\O$ 
	(observe that $B_r(x_j)$ is contained in a ball
	of radius $\rho$ centred at some point on $\partial\O$).
	Then, in any case, we infer that
	$$\inf_{B_r(x_j)\cap\O}u_\infty\geq C\sup_{B_r(x_j)\cap\O}u_\infty.$$
	Now, for any $i\in\{2,\dots,n\}$, by~\eqref{connected} there
	exists $j<i$ for which $B_r(x_i)\cap B_r(x_j)\cap\O\neq\emptyset$ and therefore
	$$\inf_{B_r(x_i)\cap\O}u_\infty\geq C \sup_{B_r(x_i)\cap\O}u_\infty\geq
	 C \inf_{B_r(x_j)\cap\O}u.$$
	Iterating at most $i-2$ times we eventually deduce from~\eqref{supOR} that
	$$\inf_{B_r(x_i)\cap\O}u_\infty\geq C^{i-1}\inf_{B_r(x_1)\cap\O}u_\infty
	\geq C^i\sup_{B_r(x_1)\cap\O}u\geq C^i\sup_{\O_R(y)}u.$$
	This implies that $\inf_{\O_R(y)}u_\infty\geq
	C^n\sup_{\O_R(y)}u_\infty$, whence by~\eqref{lbreldense}, $\inf_{\O}u_\infty\geq C^n\e$.
\end{proof}

\begin{proof}[Proof of Table~\ref{tab:attraction}]
	Let us prove the implication of the first line. The condition~$\lB>0$ 
	implies the existence of a positive supersolution $\phi$ 
	to~\eqref{epB} for some $\lambda>0$. Hence $\bar u(t,x):=\phi(x)e^{-\lambda t}$
	satisfies
	$$\partial_t \bar u=-\lambda\bar u\geq\L\bar u,$$
	and therefore $\bar u$ is a supersolution to~\eqref{BP} due to~\eqref{KPP}.
	Next, for a given solution~$u$ with a compactly supported initial datum~$u_0\geq0$, 
	we can take $\tau$ sufficiently
	negative so that $\bar u(\tau,x)\geq u_0(x)$ for all $x\in\O$.
	It follows from the comparison principle that
	$$0\leq\lim_{t\to+\infty}u(t,x)\leq\lim_{t\to+\infty}\bar u(t,x)=0,$$
	locally uniformly with respect to~$x\in\ol\O$.
	
	The same argument works for a general bounded initial datum~$u_0\geq0$,
	provided the supersolution $\phi$ has a positive infimum. This proves the second 
	row of~the~table.
	
	Suppose now that \eqref{KPP}-\eqref{saturation} hold and that
	$\mu_{\mc{B}}^b>0$.
	Let $v$ be the solution to~\eqref{BP} 
	with a constant initial datum $v_0$ larger than the value $S$ from~\eqref{saturation}.
	Because $v_0$ is a supersolution to~\eqref{BP}, the comparison principle implies that
	$v$ is nonincreasing in $t$ and thus, as $t\to+\infty$, it converges locally uniformly to a 
	stationary solution $\phi$ of~\eqref{BP} satisfying $0\leq\phi\leq v_0$.
	We necessarily have that $\phi\equiv 0$, because otherwise $\phi$ would be everywhere positive, by the \SMP,
	which would yield $\mu_{\mc{B}}^b\leq0$ due to~\eqref{KPP}. We deduce from the comparison principle
	that any solution $u$ with an initial datum $0\leq u_0\leq v_0$ converges locally uniformly to~$0$
	as $t\to+\infty$. Since $v_0$ can be taken arbitrarily large, this means that $0$ is globally attractive.	
\end{proof}


\section{The self-adjoint case}\label{sec:s-a}

We now focus on the case where $\L$ and $\B$ are in the self-adjoint form~\eqref{LBs-a}.
%
The estimate on the \pe\ we are going to derive rely on a geometrical 
lemma concerning the growth of the domains $\O_r$. We recall 
that~$\O_r$ denotes the connected component of $B_r\cap\O$ containing the origin.
For later use, we derive the result in terms of a measure which is absolutely continuous 
with respect to the Lebesgue measure.

%
%
\begin{lemma}\label{lem:geo}
	Let $\O\subset\R^d$ be a measurable set 
	and let $|\.|_f$ be the measure associated with a nonnegative density $f\in L^\infty_{loc}(\ol\O)$.
	Assume that $|\O_1|_f>0$.
	Then 
	$$\forall n\in\N,\qquad
	\min_{m\in\{1,\dots,n\}}\frac{|\O_{m+1}\setminus B_m|_f}{|\O_m|_f}<
	\bigg(\frac{|B_1|}{|\O_1|_f}\|f\|_{L^\infty(\O_{n+1})}\, n(n+1)^d+1\bigg)^{1/n}-1.$$
\end{lemma}

\begin{proof}

\bigskip

	Fix $n\in\N$. Our goal is to estimate the quantity
	$$k:=\min_{m\in\{1,\dots,n\}}\frac{|\O_{m+1}\setminus B_m|_f}{|\O_m|_f}.$$
	For $j\in\N\cup\{0\}$, let us call 
	$$\alpha_j:=|\O_{j+1}\setminus B_j|_f,$$
	with the standard convention that $B_0=\emptyset$.
	For any integer $m\leq n$, there holds that 
	$$\O_m\supset\bigcup_{j=0}^{m-1} \O_{j+1}\setminus B_j,$$
	whence
	$$|\O_m|_f\geq\sum_{j=0}^{m-1} \alpha_j.$$ 
	Then, because $\alpha_m\geq k |\O_m|_f$ by the definition of $k$, we find that
	$$\forall m\leq n,\quad
	\alpha_m\geq k \sum_{j=0}^{m-1} \alpha_j.$$
	One can check by induction that this yields
	$$\forall m\leq n,\quad
	\alpha_m\geq \alpha_0k(1+k)^{m-1}.$$
	From this, in order to get an explicit inequality for $k$, 
	we call $g(x):=x(1+x)^{n-1}$ and observe that $g'(x)\geq (1+x)^{n-1}$ for $x\geq 0$, whence
	$$\forall x\geq 0,\quad
	g(x)\geq \frac{(1+x)^n-1}{n}.$$
	We therefore derive $\alpha_n\geq\alpha_0\frac{(1+k)^n-1}n$, that is,
	$$k\leq \Big(n\frac{\alpha_n}{\alpha_0}+1\Big)^{1/n}-1,$$
	which gives the desired estimate because $\alpha_n< |\O_{n+1}|_f\leq|B_1|(n+1)^d\|f\|_{L^\infty(\O_{n+1})}$.
\end{proof}

The crucial consequence of this lemma is that the right-hand side of the inequality tends to $0$
as $n\to+\infty$ provided $f$ is bounded (or, more in general, it has 
sub-exponential growth). Of course the lemma holds true if one considers balls centred 
at an arbitrary point $y\in\O$ rather than the origin.
For uniformly smooth domains, we can get rid of the term~$|\O_1|$ 
and obtain an estimate which is uniform with respect to the centre of the balls.

\begin{lemma}\label{lem:C}
Let $\O\subset\R^d$ be an open set 
containing the origin, with $\partial\O$ locally of class $C^1$ and
satisfying the uniform interior ball condition of some radius $\rho>0$.
Then, 
there holds 
$$|\O_1|\geq |B_1|(\min\{\rho,1/2\})^d.$$

Furthermore, if $\O$ is connected and unbounded then
$$\forall r>0,\quad
|\O_{r+1}\setminus B_r|\geq |B_1|(\min\{\rho,1/4\})^d.$$
\end{lemma}

\begin{proof}
	Let us call 
	$$\t\rho:=\min\{\rho,1/2\},\qquad
	\delta:=\dist(0,\partial\O).$$
	If $\delta\geq\t\rho$ then the result trivially holds. Suppose that $\delta<\t\rho$.
	We have that $B_\delta\subset\O$ and then  there exists $\xi\in\partial B_\delta\cap\partial\O$.   
	We know that $\O$ satisfies the interior ball condition of radius $\rho$.
	Clearly, this holds true with $\rho$ replaced by $\t\rho$.
	This means that there exists $y\in\O$ such that $B_{\t\rho}(y)\subset\O$ and 
	$\xi\in\partial B_{\t\rho}(y)$.
	Because $\partial\O$ is of class~$C^1$, the balls $B_\delta$ and $B_{\t\rho}(y)$
	must be tangent, and actually $B_\delta\subset B_{\t\rho}(y)$ 
	(because $\delta<\t\rho$).
	In particular, we have that 
	$|y|=\t\rho-\delta<\t\rho$, which implies that $B_{\t\rho}(y)\subset B_{2\t\rho}\subset B_1$.
	As a consequence, we find that
	$B_{\t\rho}(y)\subset\O_1$.
	The first statement then follows.
	
	The second statement is derived in a similar way. Take $r>0$.
	Using the fact that $\O$ is unbounded and connected, 
	it is easily seen that $\O_{r+1}$ must intersect the set $\partial B_{r+\frac12}$.
	Let $z\in \O_{r+1}\cap \partial B_{r+\frac12}$. We let
	$\delta$ be the maximal radius of the balls centred at~$z$ and contained in $\O$.
	Then we distinguish the cases $\delta\geq \t\rho$ and $\delta<\t\rho$, but this time with
	$\t\rho:=\min\{\rho,1/4\}$. In the first case we are done.
	In the second one we proceed as before and we find
	$y\in\O$ such that 
	$B_\delta(z)\subset B_{\t\rho}(y)\subset\O$. From this we deduce that
	$$\left||y|-\Big(r+\frac12\Big)\right|\leq|y-z|=\t\rho-\delta<\t\rho. $$
	Therefore, $B_{\t\rho}(y)\subset B_{r+1}\setminus B_r$ because 
	$\t\rho\leq1/4$. It follows in particular that 
	$B_{\t\rho}(y)\subset \O_{r+1}\setminus B_r$, whence the conclusion of the lemma.
\end{proof}

\begin{remark}\label{rem:BHNgeneral}
Applying Lemma~\ref{lem:geo} with $\O=\R^d$ and $f=\1_\O$, and then 
Lemma~\ref{lem:C}, we derive, for every $y\in\O$,
$$\forall n\in\N,\qquad
	\min_{m\in\{1,\dots,n\}}\frac{|\O\cap B_{m+1}(y)\setminus B_m(y)|}{|\O\cap B_m(y)|}<
	\bigg(C n(n+1)^d+1\bigg)^{1/n}-1,$$
with $C$ only depending on the radius $\rho$ of the uniform interior ball condition.
Then, taking the limit as $n\to\infty$ in both sides and using the
last statement of Lemma~\ref{lem:C}, we derive
\Fi{b'}
\liminf_{n\to\infty}\frac{|\O\cap B_{n+1}(y)|}{|\O\cap B_n(y)|}=1,\quad
\text{uniformly with respect to $y\in\O$}.
\Ff
Such condition, with ``$\liminf$'' replaced by ``$\lim$'', is one of
the geometric assumptions of \cite[Theorem~1.7]{BHNgeneral}, c.f.~hypothesis $b)$ in
the statement of that theorem that we reclaim here as \thm{BHNgeneral}.
What we have just shown is that a weaker version of such hypothesis is guaranteed by
the uniform regularity of the domain, and this will be enough to prove the hair-trigger effect.  
Instead, the stronger hypothesis $b)$, even restricted at one fixed point $y$,
may fail in a uniformly smooth domain, as shown for instance by the domain
$$\O:=\mc{C}\cup\bigcup_{n\in\N}B_{2^n+1}\setminus\ol B_{2^n},$$
where $\mc{C}$ is an approximation of a cylinder (whose role is just to make $\O$ connected).
We also mention that there exist some smooth, but not uniformly smooth, domains 
for which~\eqref{b'} fails.
\end{remark}


The previous geometrical lemmas will allow us to estimate the \pe s in truncated domains
thanks to a Rayleigh-Ritz-type formula.
We point out that the arguments in~\cite{BHNgeneral} also make use of 
a Rayleigh quotient in truncated domains, but without invoking 
the \pe, in order to avoid the difficulties of its construction due to 
the lack of regularity of the boundary.

\begin{proposition}\label{pro:RR}
	Let $\L,\B$ be given by~\eqref{LBs-a}. Then, for a.e.~$r>0$, there holds that
	\Fi{RR}
	\l(y,r)=\min_{\su{v\in H^1(\O_r(y)),\ v\not\equiv0}{\tr v=0\text{ on }\partial B_r(y)}}
	\frac{\int_{\O_r(y)}(\nabla v\.A\nabla v-c v^2)+\int_{B_r(y)\cap \partial \O_r(y)}\gamma v^2}
	{\int_{\O_r(y)} v^2}.
	\Ff
	
	Moreover,
	$$
	\lB=\inf_{\su{v\in C^1_0(\ol\O)}{v\not\equiv0}}
	\frac{\int_{\O}(\nabla v\.A\nabla v-c v^2)+\int_{\partial \O}\gamma v^2}
	{\int_{\O} v^2}.
	$$
\end{proposition}
We recall that $C^1_0(E)$ denotes  the set of $C^1$ functions with compact support in~$E$ and
$\tr$ stands for trace operator.
If $A,\gamma$ are bounded then one can replace $C^1_0(\ol\O)$ with $H^1(\O)$
with null trace on the boundary in the formula for~$\lB$.
The proof of this Rayleigh-Ritz formula is rather standard, although some adaptation is 
needed because the truncated 
domain is not smooth
(for instance, regularity estimates on the derivatives
that would allow one to 
plug the \pf\ of \thm{lrinfty} into~\eqref{RR} 
become very involved, c.f.~\cite{Lie89})
We give it in the appendix.

\begin{proof}[Proof of \thm{l<<0}]
%
	Take $y\in\O$ and $n\in\N$. It follows from Lemma~\ref{lem:geo} (with $f\equiv1$)
	that there exists
	$m=m(y,n)\in\{1,\dots,n\}$ such that
	\Fi{myn}
	\frac{|\O_{m+1}(y)\setminus B_{m}(y)|}{|\O_{m}(y)|}<
	\bigg(\frac{|B_1(y)|}{|\O_1(y)|}n(n+1)^d+1\bigg)^{1/n}-1.
	\Ff
	\thm{lrinfty} and Proposition~\ref{pro:RR} provide us with a radius $r\in(m+\frac12, m+1)$ 
	for which the mixed problem~\eqref{ep-mixed} admits
	the \pe~$\lambda(y,r)$ satisfying~\eqref{RR}. 
%
	Let $v$ be a smooth function compactly supported in $B_r(y)$ and satisfying
	$ v\equiv1$ on $B_{m}(y)$ and $0\leq v\leq1$,  $|\nabla v|\leq2$ elsewhere.
	We deduce from~\eqref{RR} (here $\gamma\equiv0$) that
	$$\l(y,r)\leq
	4\bar A \,\frac{\left|\O_r(y)\setminus B_{m}(y)\right|}
	{\left|\O_r(y)\cap B_{m}(y)\right|}
	-\frac{\int_{\O_r(y)}c v^2}
	{\int_{\O_r(y)} v^2},$$
	where $\bar A$ is the largest ellipticity constant of $A$.
 	We rewrite the second term of the right-hand side as
	\[
	\frac{\int_{\O_r(y)}c v^2}
	{\int_{\O_r(y)} v^2} =
	\frac{\int_{\O_r(y)}c}
	{\left|\O_r(y)\right|}
	\frac{\left|\O_r(y)\right|}
	{\int_{\O_r(y)} v^2}
	+\frac{\int_{\O_r(y)}c( v^2-1)}
	{\int_{\O_r(y)} v^2},\]
	whence 
	\[\begin{split}
	\frac{\int_{\O_r(y)}c v^2}
	{\int_{\O_r(y)} v^2} - \frac{\int_{\O_r(y)}c}
	{\left|\O_r(y)\right|}
	&\geq -2\|c\|_\infty\,\frac{\int_{\O_r(y)}(1- v^2)}
	{\int_{\O_r(y)} v^2}\\
	&\geq -2\|c\|_\infty\,\frac{\left|\O_r(y)\setminus B_{m}(y)\right|}
	{\left|\O_r(y)\cap B_{m}(y)\right|}.
	\end{split}\]
	It follows that
	\[
	\l(y,r) \leq - \frac{\int_{\O_r(y)}c}
	{\left|\O_r(y)\right|}+
	(4\bar A+2\|c\|_\infty) \,\frac{\left|\O_r(y)\setminus B_{m}(y)\right|}
	{\left|\O_r(y)\cap B_{m}(y)\right|}.\]
	The last quotient above can be estimated using~\eqref{myn}, 
	because $\O_r(y)\cap B_{m}(y)\supset \O_m(y)$. As a consequence,
	we have shown the existence of an integer 
	$m(y,n)\leq n$ satisfying~\eqref{myn} and a radius
	$m(y,n)<r(y,n)\leq m(y,n)+1$ for which
	 \Fi{lyr<}
	 \l(y,r(y,n))\leq - \frac{\int_{\O_{r(y,n)}(y)}c}{\left|\O_{r(y,n)}(y)\right|}+
	(4\bar A+2\|c\|_\infty)
	\left[\bigg(\frac{|B_1(y)|}{|\O_1(y)|}n(n+1)^d+1\bigg)^{1/n}-1\right].
	\Ff
	We now observe that the right-hand side of~\eqref{myn} tends to $0$ as 
	$n\to\infty$, whence $m(y,n)$ tends to $\infty$
	as $n\to\infty$. Thus, for any $\e>0$, choosing $n$ large enough,
	we have from one hand that the second term in the right-hand side
	of~\eqref{lyr<} is smaller than
	$\e/2$, and from the other that $r(y,n)>m(y,n)$ is sufficiently large to have that the first term 
	is smaller than $-\mean{c}+\e/2$ (recall the definition of the average~$\mean{\.}$ from 
	Section~\ref{intro:s-a}).
	Therefore, with this choice, we derive $\l(y,r(y,n))\leq-\mean{c}+\e$.
	The first statement of the theorem then follows from \thm{lrinfty} and the arbitrariness of $\e$.
	
	To prove the second statement, we apply Lemma~\ref{lem:C}.
	The first part of the lemma implies that 
	the right-hand side of~\eqref{myn} converges to~$0$ as
	$n\to\infty$ uniformly with respect to $y\in\O$.
	Then, the second part implies that the integer $m(y,n)$ in~\eqref{myn}
	tends to $\infty$ as $n\to\infty$ uniformly with respect to $y\in\O$, 
	and thus the same occurs for $r(y,n)>m(y,n)$. 
	Recalling the definition of~$\lm{\.}$, this implies that the first
	term in the right-hand side of~\eqref{lyr<} is smaller than $-\lm{c}+\e/2$ 
	for~$n$ sufficiently large, independent of $y$.	
	Therefore, for any $\e>0$, choosing $n$ large enough we deduce from~\eqref{lyr<}
	that $\l(y,r(y,n))\leq-\lm{c}+\e$, for any $y\in\O$ and for some 
	$m(y,n)<r(y,n)<m(y,n)+1\leq n+1$. Hence, by \thm{lrinfty},
	$\l(y,n+1)\leq-\lm{c}+\e$, for any $y\in\O$ (recall that $\l(y,\cdot)$ is extended as a 
	decreasing function to the whole $\R_+$). We then derive~$\LB\leq-\lm{c}+\e$,
	which concludes the proof.
\end{proof}

 \begin{remark}\label{rem:l<<}
 	If $\O$ is bounded then there exists $R>0$ such that
	$\O_r(y)=\O$ for any~$y\in\O$ and $r\geq R$.
	It follows that $\lm{c}$ coincides with the average of
	$c$ on $\O$ and that $\l(y,r)$ coincides, for any $r>R$, with the classical Neumann \pe\
	of $-\L$ in $\O$. 
	Therefore, in such case, \thm{l<<0} reduces to the well known property 
	that the Neumann \pe\ is smaller than or equal to the average of $c$.
\end{remark}

When the hair-trigger effect holds and the problem admits a unique steady state 
with positive infimum, one gets a precise convergence result. For instance, for the 
Fisher-KPP problem
\Fi{BPlogistic}
\begin{cases}
	\partial_t u=\nabla\cdot (A(x) \nabla u)+c(x)f(u), & \quad t>0,\ x\in\O\\
	\nu\.A(x)\nabla u=0, & \quad t>0,\ x\in\partial\O,
\end{cases}
\Ff
we have the following.

\begin{proposition}\label{pro:h-t}
Assume that $A,c$ are bounded, with $\inf_\O \ul A>0$ and $\inf_\O c>0$, 
that~$f$ is positive in $(0,1)$ and negative outside $[0,1]$ and
that $\partial\O$ is 
	uniformly of class~$C^2$.
	Then any solution to~\eqref{BPlogistic} with an initial condition $u_0\geq0,\not\equiv0$
	converges to $1$ as $t\to+\infty$ locally uniformly in~$\ol\O$.
\end{proposition}

\begin{proof}
	The linearised operator around $0$ and the boundary operator for the problem
	are given by~\eqref{LBs-aN}. Since
	$\lm{c}\geq\inf_\O c>0$, \thm{l<<0} implies that $\LB<0$.
	Thus, by \thm{LambdaB}, $0$ is uniformly repulsive.
	This means that any solution of~\eqref{BPlogistic} with a bounded initial datum $u_0\geq0,\not\equiv0$
	satisfies
	$$k:=\inf_{x\in\O}\Big(\liminf_{t\to+\infty}u(t,x)\Big)>0.$$
	On the other hand, by comparison with the solution of the ODE
	$$v'=\Big(\inf_\O c\Big)v(1-v),\quad t>0,$$
	with initial datum $v(0)=\max\{1,\sup_\O u_0\}$, which is a supersolution to~\eqref{BPlogistic},
	one~gets
	$$\sup_{x\in\O}\Big(\limsup_{t\to+\infty}u(t,x)\Big)\leq\limt v(t)=1.$$
	
	Next, consider an arbitrary sequence $\seq{t}$ diverging to $+\infty$. Then, by local boundary
	estimates, the sequence $u(t+t_n,x)$ converges (up to subsequences),
	locally uniformly in $(t,x)\in\R\times\ol\O$, to a solution $\t u$ of~\eqref{BPlogistic}
	for all $t\in\R$. We know that $k\leq\t u\leq1$. We need to show that $\t u\equiv1$.
	Let $v(t)$ be the solution of the same ODE as before, but now with $v(0)=k$.
	It converges to $1$ as $t\to+\infty$.
	Comparing $v(t)$ with the function
	$\t u(t-n,x)$, for any given $n\in\N$,
	we derive
	$$\forall t>-n,\quad
	v(t+n) \leq\t u(t,x),$$
	from which, letting $n\to\infty$, we finally obtain $\t u\geq1$.	
	This concludes the proof.
\end{proof}

\begin{proposition}\label{pro:<c>}
	There exists $c$ with $\mean{c}>0$ such that 
	$0$ is locally repulsive but not uniformly repulsive for the problem
	\Fi{linearc}
	\partial_t u=\partial_{xx}u+c(x)u-u^2,\quad t>0,\ x\in\R.
	\Ff
\end{proposition}

\begin{proof}
	We define $c$ by
	$$c(x)=1-2\sum_{n\in\N}\1_{[2^n,2^n+n]}(x).$$
	It is readily seen that $\mean{c}=1$ (whereas $\lm{c}=-1$). 
	Then $0$ is locally repulsive thanks to \thm{l<<0} and Table~\ref{tab:stability}.
	Let $u$ be the solution
	of~\eqref{linearc} with an initial datum $0\leq u_0\leq1$ supported on $\R_-$.
	It satisfies $0\leq u\leq 1$ for all times.
	The~function 
	$$w_n(x):=e^{-(x-2^n)}+e^{x-2^n-n}$$
	is a supersolution to~\eqref{linearc} in the interval $(2^n,2^n+n)$, which is larger
	than $1$ on the boundary, thus, by comparison, 
	$u(t,x)\leq w_n(x)$ for all $t>0$, $x\in(2^n,2^n+n)$.
	We deduce, for any $n\in\N$, that
	$$\forall t>0,\quad
	u(t,2^n+n/2)\leq 2e^{-n/2}.$$
	This implies that $0$ is not uniformly repulsive.	
\end{proof}


\section{Further properties and relations between the different notions}

The following observation is useful for some perturbations of supersolutions that
we will perform in the sequel.
\begin{lemma}\label{lem:w}
	Assume that $A,b,c,\gamma$ are bounded
	and that~\eqref{beta>0} holds.
	There exists a function $\bar w\!\in\! C^{2,\alpha}_{loc}(\ol\O)$ satisfying
	$$\inf_\O \bar w>0,\qquad
	\sup_\O\L \bar w<+\infty,\qquad
	\inf_{\partial\O}\B \bar w>0.$$	
\end{lemma}

\begin{proof}
	We define
	$$\bar w(x):=1+\chi(k d_\O(x)),$$
	where $d_\O$ is the signed distance from $\partial\O$, positive inside $\O$,
	$\chi:\R\to\R$ is a smooth, nonnegative function supported in $[-1,1]$ such that 
	$\chi'(0)=-1$, and $k$ is a positive constant. 
	Owing to the uniform regularity of~$\partial\O$, \cite[Lemma~14.16]{GT} implies
	that~$d_\O$ is uniformly $C^2$ in a neighbourhood of $\partial\O$ of the form
	$V:=\{x\in\R^d\ :\ |d_\O(x)|<\rho\}$. Moreover, because $\partial\O$
	is locally of class $C^{2,\alpha}$, we know from~\cite{LiNi} that 
	$d_\O\in C^{2,\alpha}_{loc}(V)$.
	Then, for $k>1/\rho$, such properties hold true for $\bar w$ in the whole $\R^d$.
	On $\partial\O$ there holds $\B\bar w\geq \gamma(1+\chi(0))+k\beta\.\nu$, 
	which has a positive infimum for $k$ large enough.
\end{proof}

Lemma~\ref{lem:w} implies in particular that $\lB^p$ is well defined under the assumption
\eqref{beta>0} (as well as it is under the assumption~\eqref{gamma>0}).

\begin{proof}[Proof of \thm{relations}]
	Statement $(i)$ is an immediate consequence of the definitions.
	Let us prove the other ones.
	
	 $(ii)$.\\
	The second row of 
	Table \ref{tab:stability} yields $\mu_\B^b\geq0$ whenever $\lB^{p,b}>0$. 
	Applying this result to the operator $\L+\l$, whose \pe s are reduced by 
	$\l$ with respect to the original ones, we deduce that $\mu_\B^b\geq\l$ 
	for any $\l<\lB^{p,b}$, that is, $\mu_\B^b\geq\lB^{p,b}$.
	
	As before, in order to derive $\mu_\B^b\leq\lB$, it is sufficient to show that
	$\lB<0$ implies $\mu_\B^b\leq0$. 
	Then assume that $\lB<0$. By \thm{lrinfty}, 
	we can find $R>0$ large enough for which $\l(0,R)<0$. We now modify the zeroth order 
	coefficients of $\L$ and~$\B$ outside~$B_R$ by taking two functions $\t c$, $\t \gamma$
	satisfying
	$$\t c\leq c\inn\O,\qquad
	\t c=c\inn\O\cap B_R,\qquad \t c=\min\big\{\inf_\O c,0\big\}\inn\O\setminus B_{R+1},$$
	$$\t \gamma\geq \gamma\on\partial\O,\qquad
	\t \gamma=\gamma\on(\partial\O)\cap B_R,\qquad 
	\t \gamma=\max\big\{\sup_{\partial\O}\gamma,0\big\}\on(\partial\O)\setminus B_{R+1}.$$
	Let $\t\L$, $\t\B$ be the associated operators and let $\lambda_{\t\B}$ be the corresponding
	generalised \pe. 
	Since $\t\L$, $\t\B$ coincide with $\L$, $\B$ when restricted to $B_R$,
	\thm{lrinfty} yields $\lambda_{\t\B}<\l(0,R)<0$.
	Let $\t\vp$ be a generalised \pf\ associated with~$\lambda_{\t\B}$, provided by \thm{ef}.
	We claim that~$\t\vp$ is bounded, which may not be the case for the 
	generalised \pf s associated with $\lB$. We reclaim from the 
	construction in the proof of \thm{lrinfty} that $\t\vp$ is obtained as the locally uniform
	limit in $\ol\O$ of a sequence of $r\to+\infty$ of the \pf s~$\t\vp_r$ of the mixed problem in the 
	truncated domains $\O_r$, normalised by $\t\vp_r(0)=1$. For $r>R+1$, these functions satisfy 
	in the sets $\O_r\setminus\ol B_{R+1}$,
	\[\begin{cases}
	-(\t\L+\t\lambda(0,r))\t\vp_r=0 & \text{ in }\O_r\setminus\ol B_{R+1}\\
	\t\B \t\vp_r=0 & \text{ on } (\partial \O_r)\setminus(\ol B_{R+1}\cup\partial B_r),
	\end{cases}\]
	with $\t\lambda(0,r)<\lambda(0,R)<0$. 
	Because $\t c+\t\lambda(0,r)<0$ and 
	$\t\gamma\geq0$ outside $B_{R+1}$,
	the elliptic \MP\ and Hopf's lemma imply
	that the maximum of $\t\vp_r$ on $\ol\O_r\setminus B_{R+1}$
	is attained either on $\partial B_{R+1}$ or on $\partial B_r$, the latter case being excluded 
	because $\t\vp_r$ vanishes there.
	This shows that the family $(\t\vp_r)_{r\geq R+1}$ is globally uniformly bounded, whence
	its limit $\t\vp$ is bounded too.
	In conclusion, $\t\vp$ is bounded and satisfies
	$$\L\t\vp\geq\t\L\t\vp=-\lambda_{\t\B}\t\vp>0\inn\O,\qquad
	\B\t\vp\leq\t\B\t\vp=0\on\partial\O,$$
	which implies that $\mu_\B^b\leq 0$.
	
	$(iii)$. \\
	We need to show that $\lB\leq\mu_\B^b$.
Take $\lambda>\mu_\B^b$. By definition, there exists $\phi$ such that
 $$\phi>0 \ \text{ in }\O,\qquad\sup_\O\phi<+\infty,\qquad
 (\L+\lambda)\phi\geq0 \inn\O,
 \qquad\B\phi\leq0\on\O.$$
 The idea is to apply the Rayleigh-Ritz formula~\eqref{RR} to a suitable cutoff of $\phi$ and
 then use the growth Lemma~\ref{lem:geo} to control the cutoff term. 
  We assume without loss of generality that $0\in\O$.
 Applying Lemma~\ref{lem:geo} with $f=\phi^2$ we infer the existence, 
 for any $\e>0$, of some $m\in\N$ such~that
\Fi{phi2}
\frac{\int_{\O_{m+1}\backslash B_m} \phi^2}{\int_{\O_m}\phi^2}<\e.
\Ff
Let $r\in(m+\frac12, m+1)$ be such that \thm{lrinfty} and Proposition~\ref{pro:RR} apply.
Consider a nonnegative smooth function $\chi$ defined on $\R^d$, satisfying
 $$\supp \chi\subset B_r,\qquad
 \chi=1\ \text{ in }B_m,\qquad |\nabla \chi|\leq2\inn\R^d.$$ 
We can take $v=\phi\chi$ in the Rayleigh-Ritz formula~\eqref{RR} and infer that
%
%
 \[\begin{split}
 \l(0,r) &\leq \frac{\int_{\O_r}\left[\nabla(\phi\chi)\.A\nabla(\phi\chi)- c\phi^2\chi^2\right]}
 {\int_{\O_r}\phi^2\chi^2}\\
 &= \frac{\int_{\O_r}\left[\chi^2\nabla\phi\.A\nabla\phi+\phi^2\nabla\chi\.A\nabla\chi+
 2\phi\chi\nabla\chi\.A\nabla\phi- c\phi^2\chi^2\right]}
 {\int_{\O_r}\phi^2\chi^2}.\\
 &= \frac{\int_{\O_r}\left[\nabla(\phi\chi^2)\.A\nabla\phi+
 \phi^2\nabla\chi\.A\nabla\chi- c\phi^2\chi^2\right]}
 {\int_{\O_r}\phi^2\chi^2}.
 \end{split}\]
Applying the divergence theorem (recall that $\O_r$ is a Lipschitz open set
 for any~$r$ for which \thm{lrinfty} applies) we get
 \[\begin{split}
 \l(0,r) &\leq \frac{\int_{\O_r}\left[(-\L\phi)\phi\chi^2
 +\phi^2\nabla\chi\.A\nabla\chi\right]}
 {\int_{\O_r}\phi^2\chi^2}\\
 &\leq \lambda+\frac{ \int_{\O_r}\phi^2\nabla\chi\.A\nabla\chi} {\int_{\O_r}\phi^2\chi^2}.
 \end{split}\]
 Since $\chi=1$ in $B_m$ and $|\nabla \chi|\leq2$, there exists $C>0$ depending on
the $L^\infty$ norm of~$A$, such that
 $$\l(0,r)\leq \lambda+
 C\frac{\int_{\O\cap(B_r\backslash B_m)} \phi^2}{\int_{\O_r\cap B_m}\phi^2}.$$
Recalling that $m<r<m+1$, we eventually deduce from~\eqref{phi2} that
$\l(0,r)\leq \lambda+C\e$, whence $\lB<\lambda+C\e$ by \thm{lrinfty}.
We have derived this inequality for arbitrary 
$\l>\mu_\B^b$ and $\e>0$, with $C$ independent of both.
The inequality $\lB\leq\mu_\B^b$ is thereby proved.

$(iv)$. \\
	Assume by contradiction that $\lB^{p}>\mu_\B^b$. Then, for given 
	$\mu_\B^b<\l<\l'<\lB^{p}$, there exist two positive functions $\phi,\psi$ satisfying 
	$$\sup_\O\phi<+\infty,\qquad (\L+\lambda)\phi\geq0 \text{ in }
	\O,\qquad \B\phi\leq0\text{ on }\partial \O,$$
	$$\inf_\O\psi>0,\qquad (\L+\lambda')\psi\leq0 \text{ in }
	\O,\qquad \B\psi\geq0\text{ on }\partial \O.$$	
	We set $\t\psi:=\psi+\e\bar w$, where~$\bar w$ is the function provided by 
	Lemma~\ref{lem:w} in the case of hypothesis~\eqref{beta>0}, 
	while~$\bar w\equiv1$ if~\eqref{gamma>0} holds. Then, for $\e>0$ small 
	enough, $\t\psi$ fulfils the same conditions as $\psi$, with a possibly
	smaller $\l'$ still larger than $\l$, together~with 
	\Fi{Bphi>0}
	\inf_{\partial\O}\B\t\psi>0.
	\Ff	
	We can now proceed as in the proof of \cite[Theorem~4.2]{BR4}.
	Consider a smooth positive function $\chi:\R^d\to\R$	such that
	$\chi(x)\to+\infty$ as $|x|\to\infty$ and $\nabla\chi,D^2\chi\in L^\infty(\R^d)$.
	We define $\psi_n:=\t\psi+\frac1n\chi$ and call 
	$$k_n:=\max_{\ol\O}\frac \phi{\psi_n}.$$
	Observe that such maximum exists, it is positive, and the sequence $\seq{k}$
	is increasing and bounded from above by $\sup \phi/\inf\t\psi$, thus it is convergent.
	Let $x_n\in\ol\O$ be a point where the maximum $k_n$ is attained.
	The function $k_n\psi_n$ touches $\phi$ from above at $x_n$.
	In order to estimate the perturbation term $\frac1n\chi(x_n)$, we observe that
	$$\frac1{k_{2n}}\leq\frac{\psi_{2n}(x_n)}{\phi(x_n)}=\frac{\psi_{n}(x_n)}{\phi(x_n)}
	-\frac{1}{2n}\frac{\chi(x_n)}{\phi(x_n)}\leq \frac1{k_n}
	-\frac{1}{2n}\frac{\chi(x_n)}{\sup \phi}.$$
	Then the convergence of $\seq{k}$ implies that $\frac1n\chi(x_n)\to0$
	as $n\to\infty$. By uniform continuity, $\frac1n\chi(x)\to0$ as $n\to\infty$
	uniformly in $x\in B_\rho(x_n)$, for any $\rho>0$.
	
	Next, we see that
	$$(\L+\lambda)\psi_n\leq(\lambda-\l')\t\psi+\frac Cn(\chi+1)\inn\O,$$
	where $C$ is a constant only depending on $d$, $\lambda$ and the $L^\infty$ norms of 
	$A,b,c$,$\nabla\chi,D^2\chi$.
	It~follows that $(\L+\l)\psi_n<0$ in a neighbourhood of $x_n$ 
	for $n$ sufficiently large. Then, $x_n$ cannot lie inside $\O$ for such values of $n$,
	because otherwise we would have that the strict supersolution
	$k_n\psi_n$ touches the subsolution $\phi$ from above at $x_n$, contradicting
	the \SMP. Hence $x_n\in\partial\O$ for~$n$ large enough.
	We compute 
	$$\B\psi_n(x_n)\geq\B\t\psi -\frac Cn\big(\chi(x_n)+1\big),$$ 
	with $C$ now depending on the $L^\infty$ norms of 
	$\beta,\gamma$,$\nabla\chi$.
	For $n$ large, the above term is 
	strictly positive due to~\eqref{Bphi>0} and therefore 
	$$\beta\.\nabla(k_n\psi_n-\phi)(x_n)=\B(k_n\psi_n-\phi)(x_n)>0.$$
	This implies that $k_n\psi_n<\phi$ somewhere, 
	contradicting the definition of~$k_n$. 
\end{proof}

\begin{remark}\label{rem:relax}
	One can check looking at the proof of \thm{relations}
	that the boundedness of the coefficients can be relaxed: no assumption is required
	(besides the standing ones of Section~\ref{sec:main}) for $(i)$ as well as for the first 
	inequality
	in $(ii)$; the second inequality in $(ii)$ only requires
	$\inf_\O c>-\infty$, $\sup_{\partial\O}\gamma<+\infty$, while $(iii)$ requires these 
	same conditions plus the boundedness of the largest
	ellipticity constant of~$A$; $(iv)$ requires  
	$\sup_\O c<+\infty$, $A,b$ bounded, $\inf_{\partial\O}\gamma>-\infty$ (to apply 
	Lemma~\ref{lem:w}), $\beta$ bounded, and
	the local $C^d$ regularity of $\partial\O$ can be relaxed to $C^1$.
\end{remark}


\appendix

\setcounter{theorem}{0}
\section{The Rayleigh-Ritz formula}

\begin{proof}[Proof of Proposition~\ref{pro:RR}]
	We assume without loss of generality  that $y=0$.
	\thm{lrinfty} provides us with a set of radii $\mc{R}$ with zero measure complement in 
	$\R_+$
	on which $r\mapsto\lambda(0,r)$ is well defined and monotone. 
	The semicontinuous extension of this function 
	is continuous on some set $\mc{R}'$, with $\R_+\setminus\mc{R}'$ at most countable.
	
	Take~$r\in\mc{R}\cap\mc{R}'$.
	Consider the energy functional 
	$$\mc{E}[v]:=\int_{\O_r}(\nabla v\.A\nabla v-c v^2)+
	\int_{B_r\cap \partial \O_r}\gamma v^2$$
	(the second integral is a surface integral).
	We want to minimise $\mc{E}$ over the set
	$$\mc{H}:=\{v\in H^1(\O_r)\ :\ \tr v=0\text{ on }\partial\O_r\cap
	\partial B_r,\ \|v\|_{L^2(\O_r)}=1\}.$$
	Observe that $\mc{E}$ is well defined on $\mc{H}$ 
	because the trace operator is continuous from
	$H^1(\O_r)$ to $L^2(\partial\O_r)$
	(recall that $\partial\O_r$ is Lipschitz whenever \thm{lrinfty} applies). 
	Actually, owing to~\cite[Theorem 3.37]{McL}, it is continuous from  
	$H^{1/2+s}(\O_r)$ to $H^{s}(\partial\O_r)\subset L^2(\partial\O_r)$ for $0<s<1/2$, 
	whence by 
	the interpolation inequality for $H^{1/2+s}(\O_r)$ in terms of $H^1(\O_r)$ and $L^2(\O_r)$,
	for any $\e>0$ we find a constant $C_\e>0$ such that
	$$
	\mc{E}[v]\geq\inf_{\O_r}\ul A\|\nabla v\|_{L^2(\O_r)}^2-\|c\|_\infty\|  v\|_{L^2(\O_r)}^2
	-\|\gamma\|_\infty(\e\|\nabla v\|_{L^2(\O_r)}^2+C_\e\| v\|_{L^2(\O_r)}^2),$$ 
	where $\ul A$ is the smallest ellipticity constant of $A$. Choosing $\e$ small enough yields
	\Fi{E>}
	\forall v\in H^1\!(\O_r),\quad
	\mc{E}[v]\geq \frac12\inf_{\O_r}\ul A
	\|\nabla v\|_{L^2(\O_r)}^2-\big(\|c\|_\infty+C_\e\|\gamma\|_\infty\big)
	\|  v\|_{L^2(\O_r)}^2.
	\Ff
	This means that $\mc{E}$ is bounded from below in $\mc{H}$. We set
	$$k:=\inf_{\mc{H}}\mc{E},$$
	which coincides with the right-hand side of~\eqref{RR}.
	We need to show that $k$ is attained and that
	$k=\l(0,r)$.
	
	Let $\seq{v}$ be a minimising sequence for $\mc{E}$ on $\mc{H}$. We know from~\eqref{E>} that
	$\seq{v}$ is bounded in $H^1(\O_r)$
	and therefore it converges (up to subsequences)
	in $L^2(\O_r)$ to some function~$\psi$ satisfying $\|\psi\|_{L^2(\O_r)}=1$.
	Furthermore, 	
	being $\mc{E}$ quadratic, the parallelogram law yields, for $j,l\in\N$,
	$$\mc{E}\bigg[\frac{v_j-v_l}2\bigg]= \frac12\mc{E}[v_j]+
	\frac12\mc{E}[v_l]-\mc{E}\bigg[\frac{v_j+v_l}2\bigg]\leq
	\frac12\mc{E}[v_j]+\frac12\mc{E}[v_l]-\frac k4\|v_j+v_l\|_{L^2(\O_r)}^2,$$
	which tends to $0$ as $j,l\to\infty$. It then follows from~\eqref{E>}
	that $\seq{v}$ is a Cauchy sequence in $H^1(\O_r)$, whence
	$\psi\in \mc{H}$ and satisfies $\mc{E}[\psi]=k=\min_{\mc{H}}\mc{E}$. 
	
	Let us show that $k\leq\l(0,r)$.
	Consider a \pf\ $\vp$ associated with $\l(0,r)$, given by \thm{lrinfty}.
	Since we do not know if $\vp$ belongs to $H^1(\O_r)$, we perform a truncation.
	For $\e>0$, define $v_\e:=\max\{\vp-\e,0\}$, that is, 
	the positive part of $\vp-\e$. 
	Because $\vp$ is regular in 
	$\ol\O_{r}\setminus\partial B_{r}$, continuous in $\ol\O_r$ and
	vanishes on~$\partial B_{r}$, the function
	$v_\e$ is equal to $0$ in a neighbourhood of $\partial B_r$, hence it 
	is in $H^1(\O_r)$ with null trace on $\partial\O_r\cap\partial B_r$.
	In order to compute
	$$\mc{E}[v_\e]=\int_{\O_r}(\nabla v_\e\.A\nabla v_\e-c v_\e^2)+
	\int_{B_r\cap \partial \O_r}\gamma v_\e^2,$$
	we observe that
	$\nabla v_\e$ is (a.e.) equal to $\nabla\vp$ where $\vp>0$ and $0$
	otherwise,~whence
	$$\nabla v_\e\.A\nabla v_\e=\nabla v_\e\.A\nabla \vp=
	\nabla\.(v_\e A\nabla\vp)-v_\e\nabla\. A\nabla\vp.$$
	We can apply the divergence theorem in the Lipschitz domain $\O_r$
	to the function $v_\e A\nabla\vp$, which belongs to $H^1(\O_r)$ and has
	null trace on $\partial\O_r\cap\partial B_r$. Hence, 
	using the equation and the boundary condition satisfied by $\vp$,
	we find that
	$$\mc{E}[v_\e]=\int_{\O_r}\Big(v_\e(c+\lambda(0,r))\vp-c v_\e^2\Big)+
	\int_{B_r\cap \partial \O_r}\gamma v_\e (v_\e-\vp).$$
	By the dominated convergence theorem, 
	we eventually infer that 
	$$\lim_{\e\to0}\mc{E}\bigg[\frac{v_\e}{\|v_\e\|_{L^2(\O_r)}}\bigg]
	=\frac1{\|\vp\|_{L^2(\O_r)}^2}\lim_{\e\to0}\mc{E}[v_\e]=\lambda(0,r),$$
	whence $k\leq\l(0,r)$.

	To prove the reverse inequality, take $\t r\in\mc{R}$ 
	such that $\t r>r$ and let $\t\vp$ be an associated \pf.
	Multiplying the eigenvalue equation for $\t\vp$ by the function
	$\psi^2/\t\vp$, 
	which belongs to $H^1(\O_{r})$ and has
	null trace on $\partial\O_r\cap\partial B_r$,
	and integrating over $\O_r$, we get 	
	$$\lambda(0,\t r)\int_{\O_r}\psi^2=\int_{\O_r}\bigg(-
	\frac{\psi^2}{\t\vp}
	\nabla\.(A\nabla\t\vp)-c\psi^2\bigg).$$
	We integrate by parts the first term of the right-hand side 
	and obtain
	\[\begin{split}
	\lambda(0,\t r) &=\int_{\O_r}\bigg(
	2\frac{\psi}{\t\vp}\nabla\psi\.A\nabla\t\vp
	-\frac{\psi^2}{\t\vp^2}\nabla\t\vp\.A\nabla\t\vp-c\psi^2\bigg)
	+\int_{B_r\cap \partial \O_r}\gamma \psi^2\\
	&\leq \int_{\O_r}\big(\nabla\psi\.A\nabla\psi-c\psi^2\big)+
	\int_{B_r\cap \partial \O_r}\gamma \psi^2.
	\end{split}\]
	The right-hand side is precisely $\mc{E}[\psi]=k$.
	We have thereby shown that $\lambda(0,\t r)\leq k$
	for any $\mc{R}\ni \t r>r$,
	whence $\lambda(0,r)\leq k$
	because $r$ is in the continuity set $\mc{R'}$ of $\lambda(0,\.)$.
	
	Finally, the Rayleigh-Ritz formula for $\lB$ easily follows from the one for $\lambda(0,r)$
	owing to \thm{lrinfty}.	
\end{proof}


\section*{Acknowledgements}
This work has been supported by the ERC Advanced Grant
2013 n.~321186 ``ReaDi"
held by H.~Berestycki and by the French National Research Agency (ANR), within project NONLOCAL ANR-14-CE25-0013.
The author is grateful to N.~Nadirashvili for the suggestion of some of the questions 
investigated in this manuscript.


\tolerance = 1500 
\hoffset = .3cm
\voffset = -.6cm
 
\textwidth = 15.6cm
\textheight = 24cm
\topmargin = 0pt
\headheight = 20pt
\oddsidemargin = 0pt
\evensidemargin = 0pt
\marginparwidth = 10pt
\marginparsep = 10pt


\def\cprime{$'$} \def\polhk#1{\setbox0=\hbox{#1}{\ooalign{\hidewidth
			\lower1.5ex\hbox{`}\hidewidth\crcr\unhbox0}}}
\def\cfac#1{\ifmmode\setbox7\hbox{$\accent"5E#1$}\else
	\setbox7\hbox{\accent"5E#1}\penalty 10000\relax\fi\raise 1\ht7
	\hbox{\lower1.15ex\hbox to 1\wd7{\hss\accent"13\hss}}\penalty 10000
	\hskip-1\wd7\penalty 10000\box7}



\end{document}